\documentclass[12pt,oneside,a4paper]{article}
\usepackage{amsmath,amsthm, amssymb}
\usepackage{hyperref}
\usepackage{cleveref}

\usepackage[left=3.0cm,right=3.0cm,top=2.0cm,bottom=2.0cm]{geometry}
\usepackage[utf8]{inputenc}
\usepackage{hyperref}
\usepackage{esint}
\usepackage{stmaryrd}
\usepackage{microtype}

\newtheoremstyle{mytheorem}
  {3pt}
  {3pt}
  {\itshape}
  {}
  {\bfseries}
  {.}
  {1em}
  {}

\newtheorem{theoremA}{Theorem}

\newtheorem{theoremB}{Theorem}

\newtheorem{theoremC}{Theorem}

\newtheorem{definition}{Definition}[section]
\newtheorem{proposition}[definition]{Proposition}

\theoremstyle{mytheorem}
\newtheorem{theorem}[definition]{Theorem}

\newtheorem{lemma}[definition]{Lemma}

\newtheorem{corollary}[definition]{Corollary}

\theoremstyle{remark}
\newtheorem{remark}[definition]{Remark}

\begin{document}
\title{Cohn--Vossen-Type Inequalities for Three-Manifolds and Locally Conformally Flat Manifolds}
\author{Jialong Deng \thanks{Jialongdeng@gmail.com}}
\date{}

\maketitle

\begin{abstract}
We prove Cohn--Vossen-type scalar curvature inequalities on complete
noncompact Riemannian manifolds with nonnegative Ricci curvature, motivated by
Yau's higher-dimensional problem. For \(n\ge3\), we obtain an
\(O(r^{n-2})\) normalized growth estimate under a \(\mathbb Z^{n-2}\) subgroup
condition on \(\pi_1(M)\). For locally conformally flat manifolds, we prove
the corresponding normalized estimate in the non-\(\mathbb R^n\) case and
derive polynomial or exponential upper bounds in the conformally Euclidean
case.

In dimension three, we prove the sharp asymptotic scalar-curvature flux
estimate \(8\pi\) under quadratic scalar-curvature decay,
confirming  the Munteanu--Wang conjectural \(8\pi\) bound in this
setting; we also prove finiteness of the flux for manifolds with a foliated
end. Finally, under the Cohn--Vossen-scale scalar-growth hypothesis, we prove
weighted analogues for the weighted scalar curvature \(\mathrm{Sc}_{\alpha,\beta}\) on weighted Riemannian manifolds with nonnegative Bakry--\'Emery Ricci curvature, including
sharp distinctions between the finite-dimensional and infinite-dimensional
Bakry--\'Emery regimes.
\end{abstract}

\tableofcontents
\section{Introduction}

Motivated by the problem of extending the Cohn--Vossen
inequality~\cite{zbMATH02533557} from two dimensions to higher
dimensions, Yau poses the following question in his 1990 problem
list~\cite[Problem~9]{MR1216573}:
\begin{quote}
Given an $n$-dimensional complete manifold with nonnegative Ricci
curvature, let $B(r)$ be the geodesic ball around some point $p$.  Let
$\sigma_k$ be the $k$-th symmetric function of the Ricci tensor.  Then
is it true that $r^{-n+2k}\int_{B(r)}\sigma_k$ has an upper bound when
$r$ tends to infinity?
\end{quote}
For \(k\ge 2\) and \(n\ge 3\), Yang's thesis provides a negative answer to Yau's problem~\cite{MR3010146}.  For $k=1$, one has $\sigma_1=\mathrm{Sc}_g$.

\begin{theoremA}[Cohn--Vossen-type estimates]\label{F}
Let \((M^n,g)\) be a complete, connected, oriented, nonflat Riemannian
\(n\)-manifold, \(n\ge3\), with nonnegative Ricci curvature.  Fix
\(p\in M\), and let \(B_g(p,r)\) denote the geodesic ball of radius \(r\)
centered at \(p\).

\begin{enumerate}
\item[(i)] If \(\pi_1(M)\) contains a subgroup isomorphic to
\(\mathbb Z^{n-2}\), then, for every \(r>0\),
\[
        r^{2-n}\int_{B_g(p,r)}\mathrm{Sc}_g\,dV_g
        \le 8\pi\,\omega_{n-2},
\]
where
$\omega_{n-2} :=\mathrm{Vol}_{\mathbb E^{n-2}}\bigl(B_{\mathbb E^{n-2}}(0,1)\bigr).$

\item[(ii)] Suppose that \((M,g)\) is locally conformally flat.  If \(M\) is
not homeomorphic to \(\mathbb R^n\), then
\[
        \limsup_{r\to\infty}
        r^{2-n}\int_{B_g(p,r)}\mathrm{Sc}_g\,dV_g
        <\infty .
\]
If \(M\) is homeomorphic to \(\mathbb R^n\), then, after identifying \(M\)
with \(\mathbb R^n\) in the conformally Euclidean parametrization, there
exist \(p_0\in\mathbb R^n\), \(m\in(0,1]\), and \(C>0\), depending on \(g\),
with the following properties.

If \(0<m<1\), then there exists \(A_1>0\) depending on $g$ such that, for every \(r\ge0\),
\[
        \int_{B_g(p_0,r)}\mathrm{Sc}_g\,dV_g
        \le
        \frac{4(n-1)\omega_{n-1}}{1-2^{2-n}}\,
        C^2A_1^{n-2}(1+r)^{\frac{n-2}{1-m}} .
\]
If \(m=1\), then there exist \(A,B>0\) depending on $g$ such that, for every \(r\ge0\),

\[
        \int_{B_g(p_0,r)}\mathrm{Sc}_g\,dV_g
        \le
        \frac{4(n-1)\omega_{n-1}}{1-2^{2-n}}\,
        C^2A^{n-2}e^{B(n-2)r}.
\]

\end{enumerate}
\end{theoremA}

Theorem~\ref{F} gives a positive answer to Yau's Problem~9 for $k=1$ under
the $\mathbb{Z}^{n-2}$ subgroup condition~(i) and the locally conformally flat (LCF)
assumption~(ii).  In part~(i), the key idea is that the
$\mathbb{Z}^{n-2}$ subgroup forces the Riemannian universal cover to
split off an isometric factor $\mathbb{R}^{n-2}$~\cite{2026arXiv260114231C},
reducing the problem to the two-dimensional Cohn--Vossen inequality on
the orthogonal factor.

In part~(ii), the key point is that the Zhu--Carron--Herzlich
classification~\cite{zbMATH05042496} reduces the LCF case
to the only nontrivial possibility, namely that $(M^n,g)$ is globally conformal
to $(\mathbb{R}^n,g_E)$ with nonconstant conformal factor $u$. The remaining estimate is then governed by the
asymptotic distortion of the conformal factor. The Ma--Qing
asymptotic results~\cite{zbMATH07420360} force $u$ to decay at infinity and
hence to attain a global maximum $M=u(p_0)$, while a lower bound for
$u^{\frac{2}{n-2}}$ determines how far an intrinsic ball can reach in Euclidean
coordinates. This radius grows polynomially, of order $(1+r)^{1/(1-m)}$, when
$0<m<1$, and exponentially, of order $e^{Br}$, when $m=1$. This change in
Euclidean reach is the source of the dichotomy in the argument.

\begin{remark}
In independent work, Ma~\cite{2026arXiv260613979M} obtains an upper bound for the special case of Theorem~\ref{F}(ii) where $M$ is homeomorphic to $\mathbb R^n$, using a refined singularity estimate for nonnegative $n$-superharmonic functions.\footnote{The first version of the present paper was posted on arXiv before Ma's arXiv preprint. I thank Shiguang Ma for informing me of his result in private correspondence.}
\end{remark}

In dimension three, Yau's question has been established under additional
hypotheses: when the manifold admits a pole~\cite{MR4802637}, and when
$\mathrm{Sc}_g$ is pinched between two positive constants~\cite{2025arXiv250510520M}.
In both cases the relevant limit superior is bounded above by $8\pi$.
This led Munteanu and Wang~\cite[Conjecture~1.3]{2025arXiv250510520M} to
conjecture that $8\pi$ is the sharp upper bound for one-ended complete
three-manifolds with nonnegative Ricci curvature.  The main contribution of Theorem~\ref{three-flux} below is to identify a further condition under which the $8\pi$ upper bound holds, and to prove the finiteness of the asymptotic scalar curvature flux in a foliated end setting.

\begin{theoremB}[Theorem \ref{3D Cohn}]\label{three-flux}
Let $(M^3,g)$ be a complete, connected, oriented, nonflat Riemannian
three-manifold with nonnegative Ricci curvature.  Fix $p\in M$, and
define the asymptotic scalar curvature flux by
\[
        \mathcal{C}(g)
        :=
        \limsup_{r\to\infty}
        r^{-1}\int_{B_g(p,r)}\mathrm{Sc}_g\,dV_g.
\]
Then the following hold.
\begin{enumerate}
\item[(i)] Suppose there exist constants $C_0>0$ and $R_0>0$ such that 
$\mathrm{Sc}_g(x)\le C_0\,d_g(p,x)^{-2}$ for all $x\in M$ with
$d_g(p,x)\ge R_0$.  Then $\mathcal{C}(g)\le 8\pi$.

\item[(ii)] Suppose there exists a compact domain $\Omega\Subset M^3$
with smooth nonempty boundary, $\Omega=\overline{\operatorname{Int} (\Omega)}$,
$E:=M^3\setminus\operatorname{Int}(\Omega)$ connected, and a proper
function $f:E\to[0,\infty)$ with $f\in C^0(E)\cap C^\infty(E^\circ)$,
$f^{-1}(0)=\partial E$, $|\nabla f|_g=1$ on $E^\circ:=M^3\setminus\Omega$.
Then $\mathcal{C}(g)<\infty$.
\end{enumerate}
\end{theoremB}

Since \(\mathrm{Ric}_g\ge 0\) in dimension three, the quadratic
scalar-curvature decay in part~(i) also gives quadratic Ricci curvature
decay. Reiris's theorem~\cite[Corollary~1.1]{zbMATH06423474} then
reduces the argument, after an elementary coarea estimate, to the
Euclidean volume growth case. In this remaining case, we use Xu's
Green-function exhaustion and convert Xu's weighted
estimate~\cite[Theorem~1.2]{zbMATH07177335}, by means of the quadratic
decay and a dyadic de-weighting lemma, into an unweighted bound on
geodesic balls.  We then compare the Green-function sublevel sets with geodesic balls to obtain  the desired asymptotic bound.
  The proof of part~(ii) proceeds by foliating the
exterior region of $(M,g)$ by equidistant compact hypersurfaces
evolving along a geodesic gradient flow.  Under the assumption
$\mathrm{Ric}_g\ge 0$, the Riccati equation implies concavity of the
area Jacobian, which yields uniform control on the extrinsic geometry
of the leaves.  Combining these estimates with the Gauss equation and
the Gauss--Bonnet theorem and using the topological condition
$M\cong\mathbb{R}^3$, which contributes a total of $8\pi$, we obtain
bounds on the scalar curvature over large cylindrical shells.  The
desired asymptotic estimate then follows by comparing metric balls with
these controlled shells.

The final part of this paper extends the Cohn--Vossen-type inequalities of
the previous sections to the weighted setting. Weighted Riemannian manifolds
$(M^n,g,e^{-f}\,dV_g)$ with nonnegative $m$-Bakry--\'Emery Ricci curvature
\[
  \mathrm{Ric}_{f,m}
  :=\mathrm{Ric}_g+\nabla^2 f-\frac{1}{m}\,df\otimes df\ge 0,
  \qquad m\in(0,\infty]
\]
(the last term omitted when $m=\infty$) are central to optimal transport and
$\mathrm{RCD}$ theory. Inspired by this, the author introduces
in~\cite[Section~4]{zbMATH07342230} the weighted scalar curvature
\[
  \mathrm{Sc}_{\alpha,\beta}
  :=\mathrm{Sc}_g+\alpha\Delta f-\beta|\nabla f|^2,
  \qquad\alpha,\beta\in\mathbb{R},
\]
and studies its properties on weighted manifolds. We establish
Cohn--Vossen-type growth bounds for $\int_{B_r}\mathrm{Sc}_{\alpha,\beta}\,dV_g$
under the following assumptions.

\begin{theoremC}\label{thm:intro-weighted}
Let $(M^n,g,e^{-f}\,dV_g)$ ($n\ge 3$) be a complete, connected,
noncompact weighted Riemannian manifold satisfying $\mathrm{Ric}_g\ge 0$ and
\[
  A_p:=\limsup_{r\to\infty}r^{2-n}\int_{B_g(p,r)}\mathrm{Sc}_g\,dV_g<\infty,
\]
and let $\alpha,\beta\in\mathbb{R}$.
For every $\varepsilon>0$ there exists $r_0=r_0(\varepsilon)$ such that for
all $r\ge r_0$:
\begin{enumerate}
\item[\textup{(i)}] if $\mathrm{Ric}_{f,m}\ge 0$ for some $m<\infty$, then
$$\displaystyle\int_{B_g(p,r)}|\mathrm{Sc}_{\alpha,\beta}|\,dV_g
\le C(n,m,\alpha,\beta)(A_p+\varepsilon+1)\,r^{n-2};$$
\item[\textup{(ii)}] if $\mathrm{Ric}_{f,\infty}\ge 0$, $\alpha\le 0$, and
$\beta\ge 0$, then
$\displaystyle\int_{B_g(p,r)}\mathrm{Sc}_{\alpha,\beta}\,dV_g
\le(1-\alpha)(A_p+\varepsilon)\,r^{n-2}$;
\item[\textup{(iii)}] if $\beta>0$ and $\mathrm{Sc}_{\alpha,\beta}\ge 0$
pointwise, then
$$\displaystyle\int_{B_g(p,r)}\mathrm{Sc}_{\alpha,\beta}\,dV_g
\le C(n,\alpha,\beta)(A_p+\varepsilon+1)\,r^{n-2}.$$
\end{enumerate}
\end{theoremC}

The dichotomy between~\textup{(i)} and~\textup{(ii)}--\textup{(iii)} is
sharp: the finite-dimensional hypothesis is coercive and yields a two-sided
bound for every $(\alpha,\beta)$, whereas $\mathrm{Ric}_{f,\infty}\ge 0$
admits no sign-free absolute estimate for any $(\alpha,\beta)\ne(0,0)$,
leaving only the conditional bounds~\textup{(ii)}--\textup{(iii)}, whose
parameter range is optimal.

The proofs rest on a single localization device: a family of cutoffs
$\phi_r$ supported in $B_{2r}$ with $|\nabla\phi_r|\le C/r$, for which
$\int\phi_r^2\,\mathrm{Sc}_g$ and $\int|\nabla\phi_r|^2$ are both
$O(r^{n-2})$ by $\mathrm{Ric}_g\ge 0$, $A_p<\infty$, and the Bishop--Gromov inequality.
For~(i), tracing $\mathrm{Ric}_{f,m}\ge 0$ gives the nonnegative deficit
$Q_m:=\mathrm{Sc}_g+\Delta f-\tfrac{1}{m}|\nabla f|^2\ge 0$; testing against
$\phi_r^2$ and applying integration by parts, Cauchy--Schwarz, and Young
yields coercive $L^1$ control of $|\nabla f|^2$ and $Q_m$, which is then inserted into the algebraic identity $\mathrm{Sc}_{\alpha,\beta}$. Part~(ii) is purely
algebraic: the sign conditions force
$\mathrm{Sc}_{\alpha,\beta}\le(1-\alpha)\,\mathrm{Sc}_g$ pointwise. Part~(iii)
uses no Bakry--\'Emery hypothesis.  Instead, the pointwise condition
$\mathrm{Sc}_{\alpha,\beta}\ge 0$ supplies the coercivity, via the same
Caccioppoli argument, that $\mathrm{Ric}_{f,m}\ge 0$ provides in~(i).
Sharpness is witnessed by explicit functions $f$ on
$(\mathbb{R}^n,g_{\mathrm{Euc}})$.

\paragraph*{Organization of the Paper.}
The paper is organized around the three main results; each is proved in a separate section. In Section~\ref{LCF case}, we treat the \(n\)-dimensional case and
prove Theorem~\ref{F}, covering both the \(\mathbb Z^{n-2}\) subgroup
condition and the LCF case. In Section~\ref{dim 3}, we
specialize to dimension three and prove Theorem~\ref{three-flux},
establishing the upper bound under quadratic
scalar-curvature decay and the finiteness of the asymptotic scalar-curvature
flux in the foliated-end setting. In Section~\ref{weighted}, we develop the
weighted analogue and prove Theorem~\ref{thm:intro-weighted}, deriving
Cohn--Vossen-type growth bounds for \(\mathrm{Sc}_{\alpha,\beta}\) in the
finite- and infinite-dimensional Bakry--\'Emery regimes.

\paragraph*{Acknowledgments.}
This work is part of a broader project on LCF manifolds
that I began during my stay at the Yau  Center, where I
was supported by NSFC grant 12401063 and partially supported by NSFC grant 12271284, and
where I benefited from the support of Akito Futaki and Shing-Tung Yau. I also thank Man Chun Lee for pointing out the reference~\cite{zbMATH06423474}.

\section{Cohn--Vossen-Type Inequalities for LCF Manifolds}\label{LCF case}
 In this section, we prove Cohn--Vossen-type inequalities for locally conformally flat manifolds, giving a partial positive answer to Yau's problem. The assumption of nonnegative Ricci curvature is essential. Indeed, without it the hyperbolic $n$-space gives a counterexample, since it has constant negative Ricci curvature and exponential volume growth. For a complete asymptotically flat \(n\)-manifold of order $ \tau>0$, one can show that $r^{2-n}\int_{B_g(p,r)} \mathrm{Sc}_g\, dV_g \to 0$ as  $r \to \infty $.  We recall that  the classical Cohn--Vossen inequality \cite{zbMATH02533557} asserts that for a finitely connected, complete, oriented surface $(\Sigma,g)$, if the total integral of the sectional curvature exists, then it is bounded above by $ 4\pi\, \chi(\Sigma).$  

For a complete LCF manifold $M^n$ with nonnegative Ricci curvature,  Zhu \cite{zbMATH00537381} proves that the universal cover of such a manifold is either conformally equivalent to \(S^{n}\) or \(\mathbb{R}^{n}\), or is isometric to \(\mathbb{R} \times S^{n-1}\). Zhu also gives examples showing that for 
$(\mathbb{R}^{n},\, g = (r^{2}+1)^{-2\alpha} g_{E}), $
$r^{2} = \sum_{i=1}^{n} x_{i}^{2},$ with \(\alpha \in [0, \frac{1}{2}]\), the metric \(g\) is complete and satisfies \(\mathrm{Ric}_{g} \ge 0\).  
Moreover, when \(\alpha = \tfrac{1}{2}\), one has \(\mathrm{Vol}_{g}(B(r)) = c_{n} r\), while for \(0 \le \alpha < \tfrac{1}{2}\),
\(\mathrm{Vol}_{g}(B(r)) = c_{n} r^{n}\). More results on LCF manifolds can be found in \cite{2025arXiv251213528D}. Carron and Herzlich improve Zhu’s result and classify such manifolds in \cite{zbMATH05042496}.

\begin{theorem}[Zhu--Carron--Herzlich]
Let $(M^n,g)$ ($n\geq 3$) be a complete locally conformally flat Riemannian manifold with nonnegative Ricci curvature. Then exactly one of the
following holds:
\begin{enumerate}
\item $(M^n,g)$ is globally conformally equivalent to $(\mathbb{R}^n,g_E)$, equipped with a
nonflat metric of nonnegative Ricci curvature.
\item $(M^n,g)$ is  globally conformally equivalent to a space form of positive curvature
equipped with a conformal metric of nonnegative Ricci curvature.
\item $(M^n,g)$ is locally isometric to the cylinder $\bigl(\mathbb{R} \times S^{n-1},\, g_{\mathrm{cyl}} := dt^2 + g_{st}\bigr)$.
\item $(M^n,g)$ is isometric to a complete flat manifold.
\end{enumerate}
\end{theorem}

Thus, a closed LCF manifold \(N\) with positive Ricci curvature is conformally equivalent to a spherical space form. Note that \(N\) need not be isometric to a spherical space form, since one may perform a conformal deformation that is \(C^{\infty}\)-close to the standard round metric and has positive but nonconstant Ricci curvature.
 However, if one further assumes that the scalar curvature is constant, then \(N\) is indeed isometric to a spherical space form \cite{MR220213}.

We use the Zhu-Carron-Herzlich classification together with the topological restriction result for \(M\) due to Cucinotta--Magnabosco--Semola~\cite{2026arXiv260114231C} to prove   Theorem \ref{F}.

\subsection{Conformal metrics on Euclidean space}
The computation in Zhu's example above further suggests that Yau's problem admits a positive answer in the case \(k=1\). In fact, the following proposition treats this model situation. We include this elementary model situation because exterior radial symmetry reduces the end to a one-dimensional warped product analysis, making the expected $r^{n-2}$-growth transparent and clearly separating the symmetric case from the genuinely nonsymmetric one.

\begin{proposition}\label{global conformal to flat}
Let $(\mathbb{R}^n,g)$, $n\ge 3$, be complete with
\[
 g=u^{\frac{4}{n-2}}\,g_E,
 \qquad
 u\in C^\infty(\mathbb R^n),\quad u>0.
\]
Assume $\operatorname{Ric}_g\ge 0$ on $\mathbb R^n$ and that there exists a point
$x_0\in\mathbb R^n$ and a compact set $K\subset \mathbb R^n$ such that $u$ is
radially symmetric with respect to $x_0$ on $\mathbb R^n\setminus K$.
Let $\omega_{n-1}:=\operatorname{Vol}_{g_{st}}(S^{n-1})$. Then for every $p\in\mathbb R^n$,
\[
\limsup_{r\to\infty} r^{2-n}\int_{B_g(p,r)} \operatorname{Sc}_g\,dV_g
\le (n-1)\omega_{n-1}.
\]
\end{proposition}

\begin{proof}
After translating the Euclidean coordinates, we may assume that $x_0=0$. Choose $R_0 > 0$ and $\eta > 0$ such that $K \subset B_{\mathbb{E}}(0, R_0 - \eta)$. Since $u$ is radially symmetric about $0$ on $\mathbb{R}^n \setminus K$, we may write
$u(x) = U(|x|) $ for $|x| \ge R_0 - \eta,$
for some smooth positive function $U \colon [R_0 - \eta, \infty) \to (0,\infty)$. 

Set $V(r):=U(r)^{\frac{2}{n-2}}.$ Then on $\{|x|\ge R_0-\eta\}$, $g=V(r)^2\bigl(dr^2+r^2g_{st}\bigr).$ Define
\[
s(r):=\int_{R_0}^r V(t)\,dt,\qquad r\ge R_0-\eta.
\]
Since $V>0$, the map $r\mapsto s(r)$ is strictly increasing. We claim that
$s(r)\to\infty$  as $r\to\infty.$
Indeed, if $s(r)$ remained bounded above, then for any $\theta\in S^{n-1}$ and any
sequence $r_j\to\infty$ one would have
\[
d_g(r_j\theta,r_k\theta)\le \int_{r_j}^{r_k}V(t)\,dt
=|s(r_k)-s(r_j)|\to 0
\qquad (j,k\to\infty),
\]
so $\{r_j\theta\}$ would be a $d_g$-Cauchy sequence. Since $(\mathbb R^n,g)$ is
complete, it would converge in the metric topology, hence also in the manifold
topology; but this is impossible because $|r_j\theta|=r_j\to\infty$.
Therefore $s(r)\to\infty$  as $r\to \infty$.

Set $\sigma:=\int_{R_0-\eta}^{R_0}V(t)\,dt>0.$ Then $s:[R_0-\eta,\infty)\longrightarrow[-\sigma,\infty)$
is a smooth diffeomorphism; let $r=r(s)$ denote its inverse. Define
\[
\psi(s):=V(r(s))\,r(s),\qquad s\in[-\sigma,\infty).
\]
Then $\psi$ is smooth and positive on $[-\sigma,\infty)$, and on
$\{|x|\ge R_0-\eta\}$ the metric becomes
\[
g=ds^2+\psi(s)^2g_{st}.
\]
In particular, on $\mathbb R^n\setminus \overline{B_{\mathbb E}(0,R_0)}$
(that is, on $\{s\ge 0\}$) we have $g=ds^2+\psi(s)^2g_{st}.$
Set
\[
\Omega_0:=\overline{B_{\mathbb E}(0,R_0)},
\qquad
\rho(x):=d_g(x,\Omega_0).
\]
We show that $\rho(x)=s(|x|)$ for every $x\in \mathbb R^n\setminus \Omega_0.$ Indeed, the radial segment from $x$ to $\partial\Omega_0$ has $g$-length $s(|x|)$, so
$\rho(x)\le s(|x|).$ For the reverse inequality, let $\gamma:[0,1]\to\mathbb R^n$ be any piecewise
$C^1$ curve from $x$ to $\Omega_0$, and define $t_*:=\inf\{t\in[0,1]:\gamma(t)\in\Omega_0\}.$
Since $\Omega_0$ is closed, $\gamma(t_*)\in\Omega_0$; by minimality of $t_*$,
in fact $\gamma(t_*)\in\partial\Omega_0$. Hence $|\gamma(t_*)|=R_0$, so
$s(|\gamma(t_*)|)=0$.

Along $\gamma|_{[0,t_*]}$ we may write $\gamma(t)=(r(t),\theta(t))$, and on this
segment the metric is $ds^2+\psi(s)^2g_{st}$. Therefore,
\[
|\gamma'(t)|_g^2
=
|s'(t)|^2+\psi(s(t))^2|\theta'(t)|_{g_{st}}^2
\ge |s'(t)|^2.
\]
Thus,
\[
L_g(\gamma)\ge \int_0^{t_*}|s'(t)|\,dt
\ge \left|\int_0^{t_*} s'(t)\,dt\right|
=|s(|\gamma(t_*)|)-s(|x|)|
=s(|x|).
\]
Taking the infimum over all such curves proves the claim.

Consequently,
\[
\{\rho \le R\} = \Omega_0 \cup \{x \in \mathbb{R}^n \setminus \Omega_0 : s(|x|) \le R\},
\]
and the two sets overlap only along $\partial\Omega_0$, hence on a set of
measure zero.

For the warped product metric $g=ds^2+\psi(s)^2g_{st},$
the standard warped-product Ricci formulas give, on $(-\sigma,\infty)$,
\[
\operatorname{Ric}(\partial_s,\partial_s)=-(n-1)\frac{\psi''}{\psi},
\]
and
\[
\operatorname{Ric}|_{TS^{n-1}}
=
\left((n-2)\frac{1-(\psi')^2}{\psi^2}-\frac{\psi''}{\psi}\right)
g|_{TS^{n-1}}.
\]
Since $\operatorname{Ric}_g\ge 0$, it follows that
\[
\psi''\le 0,
\qquad
(n-2)\bigl(1-(\psi')^2\bigr)-\psi\psi''\ge 0
\quad\text{on }[0,\infty).
\]
Hence $\psi'$ is nonincreasing on $[0,\infty)$. We next show that
\[
\psi'(s)\ge 0\qquad\text{for every }s\ge 0.
\]
Indeed, if $\psi'(s_0)<0$ for some $s_0\ge 0$, then monotonicity gives
$\psi'(s)\le \psi'(s_0)<0$ for all $s\ge s_0$, and therefore
\[
\psi(s)\le \psi(s_0)+(s-s_0)\psi'(s_0)\to -\infty
\qquad\text{as }s\to\infty,
\]
contradicting $\psi>0$. Thus $\psi'$ decreases to a limit
\[
L:=\lim_{s\to\infty}\psi'(s)\in[0,\infty).
\]

We claim that $L\le 1$. Suppose instead that $L>1$. Then there exist
$\varepsilon>0$ and $S>0$ such that
\[
\psi'(s)\ge 1+\varepsilon\qquad\text{for all }s\ge S.
\]
Since $\psi'$ is nonincreasing, for $s\ge S$ we also have
\[
\psi(s)\le \psi(S)+\psi'(S)(s-S)\le Cs
\]
for some  constant $C>0$. On the other hand, the tangential Ricci inequality yields
\[
\psi''(s)
\le -\frac{(n-2)\bigl((1+\varepsilon)^2-1\bigr)}{\psi(s)}
\le -\frac{c}{s}
\qquad\text{for all }s\ge S
\]
for some $c>0$. Integrating from $S$ to $s$, we obtain
\[
\psi'(s)\le \psi'(S)-c\log\frac{s}{S}\to -\infty,
\]
contradicting $\psi'\ge 0$. Hence
\[
0\le L\le 1.
\]

We also have
\[
\lim_{s\to\infty}\frac{\psi(s)}{s}=L.
\]
Indeed,
\[
\frac{\psi(s)}{s}
=
\frac{\psi(0)}{s}
+
\frac1s\int_0^s \psi'(t)\,dt,
\]
and since $\psi'(t)\to L$, the Ces\`aro average on the right converges to $L$.

Tracing the Ricci formulas gives
\[
\operatorname{Sc}_g
=
-2(n-1)\frac{\psi''}{\psi}
+
(n-1)(n-2)\frac{1-(\psi')^2}{\psi^2}
\]
on $\{s\ge 0\}$. Define
\[
J(R):=\int_{\{\rho\le R\}} \operatorname{Sc}_g\,dV_g.
\]
Since $dV_g=\psi(s)^{n-1}\,ds\,d\mu_{g_{st}}$
on the end and $
\{\rho \le R\} = \Omega_0 \cup \{x \in \mathbb{R}^n \setminus \Omega_0 : s(|x|) \le R\}$ with only measure-zero overlap, we obtain
\[
J(R)
=
A_0
+
(n-1)\omega_{n-1}\int_0^R
\Bigl[
-2\psi''\psi^{n-2}
+
(n-2)(1-(\psi')^2)\psi^{n-3}
\Bigr]\,ds,
\]
where $A_0:=\int_{\Omega_0}\operatorname{Sc}_g\,dV_g<\infty.$ Because $\psi$ is smooth in a neighborhood of $s=0$, we may integrate by parts
directly on $[0,R]$:
\[
\int_0^R -2\psi''\psi^{n-2}\,ds
=
-2\psi'(R)\psi(R)^{n-2}
+
2\psi'(0)\psi(0)^{n-2}
+
2(n-2)\int_0^R(\psi')^2\psi^{n-3}\,ds.
\]
Hence
\begin{equation}\label{eq:global-conformal-J}
J(R)
=
A_1
+
(n-1)\omega_{n-1}
\left[
-2\psi'(R)\psi(R)^{n-2}
+
(n-2)\int_0^R(1+(\psi')^2)\psi^{n-3}\,ds
\right],
\end{equation}
where
\[
A_1:=A_0+2(n-1)\omega_{n-1}\psi'(0)\psi(0)^{n-2}.
\]

\noindent\emph{Case 1: $L>0$.}
We evaluate the limit of $R^{2-n}J(R)$ by analyzing the boundary and integral
terms in \eqref{eq:global-conformal-J} separately. In this case $\psi(s)$ has positive linear growth at infinity, since
\[
\lim_{s\to\infty}\frac{\psi(s)}{s}=L>0.
\]
Accordingly, both the boundary term and the integral term in
\eqref{eq:global-conformal-J} contribute at order $R^{n-2}$.
Thus, we have
\[
\lim_{R\to\infty} R^{2-n}\bigl[-2\psi'(R)\psi(R)^{n-2}\bigr]
=
\lim_{R\to\infty}
-2\psi'(R)\left(\frac{\psi(R)}{R}\right)^{n-2}
=
-2L^{\,n-1}.
\]

Set
\[
c:=(1+L^2)L^{\,n-3}.
\]
Since $\psi'(s)\to L$ and $\psi(s)/s \to L$, it follows that
\[
\frac{(1+(\psi'(s))^2)\psi(s)^{n-3}}{s^{n-3}}
=
(1+(\psi'(s))^2)\left(\frac{\psi(s)}{s}\right)^{n-3}
\longrightarrow c
\qquad (s\to\infty).
\]
We claim that
\[
\lim_{R\to\infty}
R^{2-n}\int_0^R (1+(\psi')^2)\psi^{n-3}\,ds
=
\frac{c}{n-2}.
\]
Indeed, fix $\delta>0$. There exists $S_\delta>0$ such that for all
$s\ge S_\delta$,
\[
(c-\delta)s^{n-3}
\le
(1+(\psi'(s))^2) (\frac{\psi(s)}{s})^{n-3}s^{n-3}
\le
(c+\delta)s^{n-3}.
\]
Integrating from $S_\delta$ to $R$ yields
\[
\int_{S_\delta}^R (c-\delta)s^{n-3}\,ds
\le
\int_{S_\delta}^R (1+(\psi')^2)\psi^{n-3}\,ds
\le
\int_{S_\delta}^R (c+\delta)s^{n-3}\,ds.
\]
After dividing by \(R^{n-2}\) and letting \(R\to\infty\), we obtain
\[
\frac{c-\delta}{n-2}
\le
\liminf_{R\to\infty}
R^{2-n}\int_0^R (1+(\psi')^2)\psi^{n-3}\,ds
\le
\limsup_{R\to\infty}
R^{2-n}\int_0^R (1+(\psi')^2)\psi^{n-3}\,ds
\le
\frac{c+\delta}{n-2},
\]
since the contribution of the fixed interval \([0,S_\delta]\) vanishes after
multiplication by \(R^{2-n}\). Letting \(\delta\downarrow 0\) proves the claim.

Dividing \eqref{eq:global-conformal-J} by \(R^{n-2}\) and letting \(R\to\infty\),
we conclude that
\[
\lim_{R\to\infty} R^{2-n}J(R)
=
(n-1)\omega_{n-1}
\left[
-2L^{\,n-1}
+
(n-2)\cdot \frac{c}{n-2}
\right].
\]
Substituting \(c=(1+L^2)L^{\,n-3}\), this becomes
\[
\lim_{R\to\infty} R^{2-n}J(R)
=
(n-1)\omega_{n-1}\Bigl[(1+L^2)L^{\,n-3}-2L^{\,n-1}\Bigr]
=
(n-1)(1-L^2)L^{\,n-3}\omega_{n-1}.
\]
Since \(0<L\le 1\), we have \((1-L^2)L^{\,n-3}\le 1\). Hence
\[
\limsup_{R\to\infty}R^{2-n}J(R)\le (n-1)\omega_{n-1}.
\]

\noindent\emph{Case 2: $L=0$ and $n>3$.}
Fix $\varepsilon>0$. Since $\psi'(s)\to 0$ and $\psi(s)/s\to 0$, there exists
$S_\varepsilon>0$ such that for all $s\ge S_\varepsilon$,
\[
0\le \psi'(s)\le \varepsilon,
\qquad
0\le \psi(s)\le \varepsilon s.
\]
Because the boundary term in \eqref{eq:global-conformal-J} is nonpositive, we may
discard it for an upper bound. Splitting the integral at $S_\varepsilon$, we obtain
\[
R^{2-n}J(R)
\le
R^{2-n}A_1
+
(n-1)\omega_{n-1}(n-2)R^{2-n}
\left[
\int_0^{S_\varepsilon}(1+(\psi')^2)\psi^{n-3}\,ds
+
\int_{S_\varepsilon}^R (1+\varepsilon^2)(\varepsilon s)^{n-3}\,ds
\right].
\]
As \(R\to\infty\), the terms involving \(A_1\) and the fixed integral over
\([0,S_\varepsilon]\) vanish. Therefore
\[
\limsup_{R\to\infty}R^{2-n}J(R)
\le
(n-1)\omega_{n-1}(1+\varepsilon^2)\varepsilon^{\,n-3}.
\]
Letting \(\varepsilon\downarrow0\), we conclude that
\[
\limsup_{R\to\infty}R^{2-n}J(R)=0.
\]

\noindent\emph{Case 3: $L=0$ and $n=3$.}
In this case \eqref{eq:global-conformal-J} becomes
\[
J(R)
=
A_1
+
2\omega_2\left[
-2\psi'(R)\psi(R)+\int_0^R (1+(\psi')^2)\,ds
\right].
\]
Again the boundary term is nonpositive. Given \(\varepsilon>0\), choose
\(S_\varepsilon>0\) such that
\[
0\le \psi'(s)\le \varepsilon
\qquad\text{for all }s\ge S_\varepsilon.
\]
Then
\[
\frac1R\int_0^R (1+(\psi')^2)\,ds
=
\frac1R\int_0^{S_\varepsilon}(1+(\psi')^2)\,ds
+
\frac1R\int_{S_\varepsilon}^R(1+(\psi')^2)\,ds
\le
\frac{C_\varepsilon}{R}+\frac{R-S_\varepsilon}{R}(1+\varepsilon^2),
\]
where
\[
C_\varepsilon:=\int_0^{S_\varepsilon}(1+(\psi')^2)\,ds.
\]
Dividing by \(R\) and letting \(R\to\infty\), we obtain
\[
\limsup_{R\to\infty}R^{-1}J(R)\le 2\omega_2(1+\varepsilon^2).
\]
Letting \(\varepsilon\downarrow0\), we get
\[
\limsup_{R\to\infty}R^{-1}J(R)\le 2\omega_2=(n-1)\omega_{n-1}.
\]

Combining the three cases, we conclude that
\begin{equation}\label{eq:global-conformal-rho}
\limsup_{R\to\infty}R^{2-n}J(R)\le (n-1)\omega_{n-1}.
\end{equation}

Finally, we pass from the exhaustion $\{\rho \le R\}$ to geodesic balls.
Fix $p\in\mathbb R^n$. Since distance to a closed set is $1$-Lipschitz,
\[
\rho(x)\le d_g(x,p)+\rho(p)
\qquad\text{for all }x\in\mathbb R^n.
\]
Hence $B_g(p,r)\subset \{ x \in \mathbb{R}^n \mid \rho(x) \le r + \rho(p) \}.$ Because $\operatorname{Ric}_g\ge 0$ implies $\operatorname{Sc}_g\ge 0$, we have
\[
\int_{B_g(p,r)} \operatorname{Sc}_g\,dV_g
\le J(r+\rho(p)).
\]
Multiplying by $r^{2-n}$ and using \eqref{eq:global-conformal-rho}, we obtain
\[
\begin{aligned}
\limsup_{r\to\infty}
r^{2-n}\int_{B_g(p,r)} \operatorname{Sc}_g\,dV_g
&\le
\limsup_{r\to\infty} r^{2-n}J(r+\rho(p)) \\
&=
\limsup_{r\to\infty}
\left(\frac{r+\rho(p)}{r}\right)^{n-2}
(r+\rho(p))^{2-n}J(r+\rho(p)) \\
&\le (n-1)\omega_{n-1}.
\end{aligned}
\]
\end{proof}

\begin{remark}
For $(\mathbb{R}^{3},\, g = (r^{2}+1)^{-2\alpha} g_{E})$ with $\alpha \in [0, \tfrac{1}{2}]$, one can show that
\[
\lim_{r \to \infty} r^{-1} \int_{B_g(p,r)} \mathrm{Sc}_g \, dV_g
= 32\pi \alpha(1-\alpha) \le 8\pi.
\]
Hence, no gap theorem holds in this setting.
\end{remark}

\begin{remark}
If $(\mathbb{R}^n, g = u^{\frac{4}{n-2}} g_{E})$ ($n \ge 3$) is scalar-flat for some $0 < u \in C^\infty(\mathbb{R}^n)$ , then $u$ is constant. Indeed, $\mathrm{Sc}_g = 0$ implies that $u$ is a positive harmonic function on all of $\mathbb{R}^n$. By the generalized Liouville theorem, any positive harmonic function on the entire Euclidean space $\mathbb{R}^n$ must be constant. Thus, $u$ is constant.
\end{remark}

\begin{remark}
Independently,  Ma \cite{2026arXiv260613979M} has recently obtained a sharp strengthening of Proposition~\ref{global conformal to flat}, without assuming that $u$ is radially symmetric, by using a refined singularity estimate for nonnegative $n$-superharmonic functions. 
\end{remark}

For a complete conformal metric
$(\mathbb{R}^n,g=u^{\frac{4}{n-2}}g_E)$ $(n\geq 3)$ with
$\mathrm{Ric}_g\geq 0$, Ma--Qing~\cite{zbMATH07420360} show that there
is an asymptotic constant $m\in[0,1]$, depending on $u$, such that if
$m=0$, then $u$ is constant. The following estimate is obtained by applying results of Ma--Qing~\cite{zbMATH07420360}.

\begin{proposition}\label{estimate}
Let $n\geq 3$, and let
$(\mathbb{R}^n,g=u^{\frac{4}{n-2}}g_E)$ $(u>0)$ be complete, smooth,
nonflat, and satisfy $\mathrm{Ric}_g\geq 0$. Then $u$ attains its global
maximum at some point $p_0\in\mathbb{R}^n$. Set
$M:=u(p_0)=\max_{\mathbb{R}^n}u$, and let $m$ be the asymptotic constant
associated with $u$.

\begin{enumerate}
\item[\rm(a)]
If $0<m<1$, then there exists $A_1>0$ such that, for every $r\geq 0$,
\begin{equation}\label{eq:scalar-bound-subcritical}
\int_{B_g(p_0,r)}\mathrm{Sc}_g\,dV_g
\leq
\frac{4(n-1)\omega_{n-1}}{1-2^{2-n}}\,
M^2A_1^{n-2}(1+r)^{\frac{n-2}{1-m}}.
\end{equation}

\item[\rm(b)]
If $m=1$, then there exist $A,B>0$ such that, for every $r\geq 0$,
\begin{equation}\label{eq:scalar-bound-critical}
\int_{B_g(p_0,r)}\mathrm{Sc}_g\,dV_g
\leq
\frac{4(n-1)\omega_{n-1}}{1-2^{2-n}}\,
M^2A^{n-2}e^{B(n-2)r}.
\end{equation}
\end{enumerate}
Here $\omega_{n-1}=|S^{n-1}|$ denotes the Euclidean area of the unit sphere
in $\mathbb{R}^n$.
\end{proposition}

The proof combines two inputs of Ma--Qing with a potential-theoretic mass
estimate for the scalar curvature. The asymptotic behavior of $u$ forces decay
at infinity and hence gives a global maximum $M=u(p_0)$. The lower bound for
the conformal factor $u^{\frac{2}{n-2}}$ converts intrinsic balls into Euclidean
balls whose radii grow polynomially, of order $(1+r)^{1/(1-m)}$, when
$0<m<1$, and exponentially, of order $e^{Br}$, when $m=1$. This is the source of the dichotomy in the argument.

\begin{proof}
Set $\phi:=\frac{2}{n-2}\log u$ and $\lambda:=e^\phi=u^{\frac{2}{n-2}}$, so that
$g=\lambda^2 g_E$. Since $g$ is nonflat, by the Ma--Qing liminf asymptotic,
\begin{equation}\label{eq:mq-liminf}
\liminf_{|x|\to\infty}\frac{\phi(x)}{\log(1/|x|)}=m>0.
\end{equation}
Fix $\eta\in(0,m)$.  From \eqref{eq:mq-liminf}, there exists $R_0>1$ such that
\[
\frac{\phi(x)}{\log(1/|x|)}>\eta
\qquad\text{whenever } |x|\ge R_0.
\]
  Thus
\[
\phi(x)<\eta\log(1/|x|)=-\eta\log |x|
\qquad\text{for } |x|\ge R_0.
\]
Therefore
\[
u(x)=\exp\!\left(\frac{n-2}{2}\phi(x)\right)
\le |x|^{-\eta(n-2)/2}\longrightarrow0
\qquad\text{as } |x|\to\infty.
\]
Since $u$ is positive and continuous, it attains a positive global maximum at
some point $p_0\in\mathbb{R}^n$.  We write $M:=u(p_0)=\max_{\mathbb{R}^n}u.$

The pointwise Ma--Qing lower bound \cite[(1.11)]{zbMATH07420360}  gives, after increasing the constant if
necessary,
\begin{equation}\label{eq:mq-lambda-lower-origin}
\lambda(x)=e^{\phi(x)}\ge c(1+|x|)^{-m}
\qquad\text{for all }x\in\mathbb{R}^n
\end{equation}
for some $c>0$.  Since
\[
1+|x|\le (1+|p_0|)(1+|x-p_0|),
\]
\eqref{eq:mq-lambda-lower-origin} implies 
\begin{equation}\label{eq:lambda-lower-at-p0}
 e^{\phi(x)}=u(x)^{\frac{2}{n-2}}
 \ge c_0(1+|x-p_0|)^{-m},
\end{equation} 
where  $c_0:=c(1+|p_0|)^{-m}>0.$  Let $\rho(y):=|y-p_0|$, and let $\gamma:[0,1]\to\mathbb{R}^n$ be any absolutely
continuous curve with $\gamma(0)=p_0$ and $\gamma(1)=x$.  Then $\rho\circ\gamma$
is absolutely continuous and satisfies
$| (\rho\circ\gamma)'(t)|\le |\gamma'(t)|_E$ for a.e. $t.$

Using \eqref{eq:lambda-lower-at-p0}, we get
\[
\begin{aligned}
L_g(\gamma)
&=\int_0^1\lambda(\gamma(t))|\gamma'(t)|_E\,d t \\
&\ge c_0\int_0^1(1+\rho(\gamma(t)))^{-m}|(\rho\circ\gamma)'(t)|\,d t.
\end{aligned}
\]
Define
\[
W(s):=\int_0^s(1+t)^{-m}\,d t.
\]
Since $W'(s)=(1+s)^{-m}$, the preceding inequality and the fundamental theorem
of calculus give
\[
L_g(\gamma)
\ge c_0\left|\int_0^1 \frac{d}{dt}W(\rho(\gamma(t)))\,d t\right|
= c_0 W(|x-p_0|).
\]
Taking the infimum over all such curves yields
\begin{equation}\label{eq:distance-lower-W}
d_g(p_0,x)\ge c_0W(|x-p_0|).
\end{equation}

If $0<m<1$, then
\[
W(s)=\frac{(1+s)^{1-m}-1}{1-m}.
\]
Thus $d_g(p_0,x)<r$ implies
\[
1+|x-p_0|<\left(1+\frac{1-m}{c_0}r\right)^{\frac{1}{1-m}}.
\]
Choosing $A_1>0$ large enough gives 

\begin{equation}\label{eq:ball-containment-subcritical}
B_g(p_0,r)
\subset
B_E\left(p_0,A_1(1+r)^{\frac{1}{1-m}}\right)
\end{equation}
 for
all $r\ge0$.  If $m=1$, then $W(s)=\log(1+s)$, and \eqref{eq:distance-lower-W}
implies
\[
|x-p_0|<e^{r/c_0}-1.
\]
Therefore, \begin{equation}\label{eq:ball-containment-critical}
B_g(p_0,r)
\subset
B_E(p_0,Ae^{Br}),
\end{equation} for example, with $A=1$ and
$B=1/c_0$.

The inner inclusion$$B_E\!\left(p_0,M^{-\frac{2}{n-2}}r\right) \subset B_g(p_0,r)$$follows from the opposite estimate. Indeed, by the global upper bound $u(x) \le M$, we have $\lambda(x)=u(x)^{\frac{2}{n-2}}\le M^{\frac{2}{n-2}}.$ The Euclidean line segment from $p_0$ to $x$ therefore gives $d_g(p_0,x)\le M^{\frac{2}{n-2}}|x-p_0|,$ which is equivalent to the inner inclusion $B_E\!\left(p_0,M^{-\frac{2}{n-2}}r\right) \subset B_g(p_0,r).$

The conformal scalar-curvature equation is
\begin{equation}\label{eq:conformal-scalar-equation}
-a_n\Delta_Eu=\mathrm{Sc}_g u^{\frac{n+2}{n-2}},
\end{equation}
where $a_n := \frac{4(n-1)}{n-2}$. Since $\mathrm{Ric}_g\ge0$, we have $\mathrm{Sc}_g\ge0$.  Hence
\[
\mu:=(-\Delta_Eu)\,d x
\]
is a nonnegative Radon measure.  Because $u\in C^\infty(\mathbb{R}^n)$, this measure is
absolutely continuous with respect to Lebesgue measure and has no atoms.  Since
$dV_g=u^{\frac{2n}{n-2}}\,d x$, equation \eqref{eq:conformal-scalar-equation}
gives
\begin{equation}\label{eq:scalar-measure-identity}
\mathrm{Sc}_g\,dV_g
=\mathrm{Sc}_gu^{\frac{2n}{n-2}}\,d x
=a_nu(-\Delta_Eu)\,d x
=a_nu\,d\mu.
\end{equation}

Let $R>0$ and put $D:=B_E(p_0,2R)$.  Let
\[
G_R(x):=\frac{1}{(n-2)\omega_{n-1}}
\left(|x-p_0|^{2-n}-(2R)^{2-n}\right),
\qquad x\in D\setminus\{p_0\}.
\]
Thus $G_R$ is the Green function of $D$ with pole at $p_0$ and zero boundary
value.  For $0<\varepsilon<R$, apply Green's second identity on
$D\setminus\overline{B_E(p_0,\varepsilon)}$:
\begin{equation}\label{eq:green-identity}
\int_{D\setminus B_E(p_0,\varepsilon)}G_R\,d\mu
=
\int_{\partial(D\setminus B_E(p_0,\varepsilon))}
\left(u\partial_\nu G_R-G_R\partial_\nu u\right)\,d S.
\end{equation}
On the outer boundary $\partial D$, one has $G_R=0$ and
\[
\partial_\nu G_R=-\frac{1}{\omega_{n-1}(2R)^{n-1}},
\]
so the outer contribution is
\[
-\frac{1}{\omega_{n-1}(2R)^{n-1}}
\int_{\partial B_E(p_0,2R)}u\,d S.
\]
On the inner boundary $\partial B_E(p_0,\varepsilon)$, the outward normal for
the punctured domain is $-\partial_r$.  Hence
\[
\partial_\nu G_R=\frac{1}{\omega_{n-1}\varepsilon^{n-1}},
\]
and therefore
\[
\int_{\partial B_E(p_0,\varepsilon)}u\partial_\nu G_R\,d S\longrightarrow u(p_0)=M.
\]
The remaining inner term tends to zero, since $G_R=O(\varepsilon^{2-n})$,
$|\nabla u|$ is bounded on $\overline{D}$, and
$|\partial B_E(p_0,\varepsilon)|=\omega_{n-1}\varepsilon^{n-1}$.
Finally, $G_R\ge0$, the punctured domains increase to $D\setminus\{p_0\}$, and
$\mu(\{p_0\})=0$; hence monotone convergence applies.  Letting
$\varepsilon\downarrow0$ in \eqref{eq:green-identity}, we obtain
\begin{equation}\label{eq:green-mu-bound}
\int_DG_R\,d\mu
=
M-\frac{1}{\omega_{n-1}(2R)^{n-1}}
\int_{\partial B_E(p_0,2R)}u\,d S
\le M.
\end{equation}
For $x\in B_E(p_0,R)$,
\[
G_R(x)\ge
\frac{R^{2-n}-(2R)^{2-n}}{(n-2)\omega_{n-1}}
=
\frac{1-2^{2-n}}{(n-2)\omega_{n-1}}R^{2-n}.
\]
Combining this lower bound with \eqref{eq:green-mu-bound} gives
\begin{equation}\label{eq:mu-ball-bound}
\mu(B_E(p_0,R))
\le
\frac{(n-2)\omega_{n-1}}{1-2^{2-n}}MR^{n-2}
\qquad\text{for all }R>0.
\end{equation}

By \eqref{eq:scalar-measure-identity} and $u(x)\le M$, one has
\begin{equation}\label{eq:scalar-by-mu-upper}
\int_{B_g(p_0,r)}\mathrm{Sc}_g\,dV_g
=a_n\int_{B_g(p_0,r)}u\,d\mu
\le a_nM\mu(B_g(p_0,r)).
\end{equation}

Assume first $0 < m < 1$. By \eqref{eq:ball-containment-subcritical} and the mass bound \eqref{eq:mu-ball-bound} with $R=A_1(1+r)^{1/(1-m)}$, the scalar curvature integral \eqref{eq:scalar-by-mu-upper} yields
\[
\int_{B_g(p_0,r)}\mathrm{Sc}_g\,dV_g
\le
\frac{4(n-1)\omega_{n-1}}{1-2^{2-n}}\,
M^2A_1^{n-2}(1+r)^{\frac{n-2}{1-m}}.
\]

If $m=1$, then \eqref{eq:ball-containment-critical}, \eqref{eq:mu-ball-bound},
and \eqref{eq:scalar-by-mu-upper} give
\[
\int_{B_g(p_0,r)}\mathrm{Sc}_g\,dV_g
\le
 a_nM\frac{(n-2)\omega_{n-1}}{1-2^{2-n}}
 M A^{n-2}e^{B(n-2)r}.
\]
\end{proof}

\subsection{Cohn--Vossen-type estimates from splitting}

In this subsection, we prove Cohn--Vossen-type estimates obtained from
splitting in the Riemannian universal cover, and then complete the proof of
Theorem~\ref{F}.

\begin{proposition}\label{codim2-splitting-Sc}
Let $(M^n,g)$ be a complete Riemannian manifold of dimension $n\ge 3$ with
$\mathrm{Ric}_g\ge 0$.
Assume that the universal cover $(\widetilde M,\widetilde g)$ splits
isometrically as
\[
(\widetilde M,\widetilde g)\cong (\Sigma^2,h)\times(\mathbb{R}^{n-2},g_E),
\]
with product metric $\widetilde g=h\oplus g_E$.
Then for every $p\in M$ and every $r>0$,
\[
r^{2-n}\int_{B_g(p,r)} \mathrm{Sc}_g\,dV_g \;\le\; 8\pi\,\omega_{n-2},
\]
where $\omega_{n-2}:=\mathrm{Vol}_{\mathbb{E}^{n-2}}(B_{\mathbb{E}^{n-2}}(0,1))$.
\end{proposition}

\begin{proof}
Fix $p\in M$ and choose a lift $\widetilde p\in \pi^{-1}(p)$, where
$\pi:(\widetilde M,\widetilde g)\to (M,g)$ is the universal Riemannian covering.
Since $\pi$ is a local isometry, we have $\widetilde g=\pi^*g$ and hence
\[
\mathrm{Ric}_{\widetilde g}=\pi^*(\mathrm{Ric}_g)\ge 0,
\qquad
\mathrm{Sc}_{\widetilde g}=\pi^*(\mathrm{Sc}_g)\ge 0,
\qquad
dV_{\widetilde g}=\pi^*(dV_g).
\]

Fix $r>0$. Let $A:=B_g(p,r)$, choose a lift $\widetilde p\in \pi^{-1}(p)$, and set
$\widetilde A:=B_{\widetilde g}(\widetilde p,r).$
Then $\pi(\widetilde A)=A.$
Indeed, if $\widetilde x\in \widetilde A$, then for any piecewise smooth curve
$\widetilde\alpha$ joining $\widetilde p$ to $\widetilde x$ we have
$\mathrm{Length}_g(\pi\circ\widetilde\alpha)=\mathrm{Length}_{\widetilde g}(\widetilde\alpha)$
(since $\pi$ is a local isometry). Taking infima over such curves gives
\[
d_g(p,\pi(\widetilde x)) \le d_{\widetilde g}(\widetilde p,\widetilde x) < r,
\]
so $\pi(\widetilde A)\subset A$.

Conversely, if $q\in A$, then $d_g(p,q)<r$. By Hopf--Rinow (completeness of $M$),
there exists a minimizing geodesic $\gamma$ from $p$ to $q$ of length
$\mathrm{Length}_g(\gamma)=d_g(p,q)<r$.
Lift $\gamma$ to a geodesic $\widetilde\gamma$ in $\widetilde M$ with
$\widetilde\gamma(0)=\widetilde p$ and set $\widetilde q:=\widetilde\gamma(1)$.
Then $\pi(\widetilde q)=q$ and
\[
d_{\widetilde g}(\widetilde p,\widetilde q)\le \mathrm{Length}_{\widetilde g}(\widetilde\gamma)
=\mathrm{Length}_g(\gamma)=d_g(p,q)<r,
\]
so $\widetilde q\in \widetilde A$, hence $q\in \pi(\widetilde A)$.
Thus, $\pi(\widetilde A)=A.$

Define the multiplicity function
\[
N_r(q):=\#\bigl(\pi^{-1}(q)\cap \widetilde A\bigr),
\qquad q\in A.
\]
Since $\pi(\widetilde A)=A$, $N_r(q)\ge 1$ for all $q\in A$.
Moreover, $N_r(q)<\infty$ for each $q\in A$. Since each fiber $\pi^{-1}(q)$ is discrete, and by the Hopf--Rinow theorem the closed ball
$\overline{B_{\widetilde g}(\widetilde p,r)}$ is compact, it follows that the intersection
$\pi^{-1}(q)\cap \overline{B_{\widetilde g}(\widetilde p,r)}$ is finite. Indeed, any discrete subset of a compact set must consist of finitely many points.

\begin{lemma}\label{coarea}
For any nonnegative measurable function $f$ on $A$, one has
\[
\int_{\widetilde A} (f\circ \pi)\, dV_{\widetilde g}
=
\int_A N_r(q)\, f(q)\, dV_g(q).
\]
\end{lemma}
\begin{proof}
Since $\overline A$ is compact, it can be covered by finitely many open sets
$\{U_i\}_{i=1}^k$ that are evenly covered by the covering map $\pi$.
Let $\{\rho_i\}_{i=1}^k$ be a smooth partition of unity subordinate to this cover,
so that $\mathrm{supp}(\rho_i)\subset U_i$ and $\sum_{i=1}^k \rho_i =1$ on $A$.

We decompose the integral over $\widetilde A$ as follows:
\[
\int_{\widetilde A} (f\circ \pi)\, dV_{\widetilde g}
=
\int_{\widetilde A} \Bigl(\sum_{i=1}^k (\rho_i\circ \pi)\Bigr)(f\circ \pi)\, dV_{\widetilde g}
=
\sum_{i=1}^k \int_{\widetilde A} ((\rho_i f)\circ \pi)\, dV_{\widetilde g}.
\]

Fix an index $i$. Since $U_i$ is evenly covered, its preimage
$\pi^{-1}(U_i)$ is a disjoint union of open sets $\{S_{i,\alpha}\}_\alpha$,
each of which is mapped isometrically onto $U_i$ by $\pi$.
Because $\rho_i$ is supported in $U_i$, the integrand vanishes outside
$\bigcup_\alpha S_{i,\alpha}$, and therefore
\[
\int_{\widetilde A} ((\rho_i f)\circ \pi)\, dV_{\widetilde g}
=
\sum_\alpha
\int_{S_{i,\alpha}\cap \widetilde A} ((\rho_i f)\circ \pi)\, dV_{\widetilde g}.
\]

On each sheet $S_{i,\alpha}$, the map $\pi$ is an isometry, so by change of variables,
\[
\int_{S_{i,\alpha}\cap \widetilde A} ((\rho_i f)\circ \pi)\, dV_{\widetilde g}
=
\int_{\pi(S_{i,\alpha}\cap \widetilde A)} \rho_i(q) f(q)\, dV_g(q).
\]

For any point $q\in U_i$, the number of sheets $S_{i,\alpha}$ intersecting
$\widetilde A$ and projecting to $q$ is exactly the multiplicity $N_r(q)$.
Summing over $\alpha$ yields
\[
\sum_\alpha
\int_{\pi(S_{i,\alpha}\cap \widetilde A)} \rho_i f
=
\int_{U_i} N_r(q)\, \rho_i(q)\, f(q)\, dV_g(q).
\]

Finally, summing over $i$ and using $\sum_i \rho_i=1$ on $A$, we obtain
\[
\sum_{i=1}^k \int_{U_i} N_r\, \rho_i f\, dV_g
=
\int_A N_r(q)\, f(q)\, dV_g(q),
\]
which completes the proof.
\end{proof}

Apply Lemma \ref{coarea} with $f=\mathrm{Sc}_g\ge 0$ and use
$\mathrm{Sc}_{\widetilde g}=\mathrm{Sc}_g\circ \pi$ to get
\begin{equation}\label{eq:pushdown-fixed}
\int_{\widetilde A} \mathrm{Sc}_{\widetilde g}\,dV_{\widetilde g}
=
\int_A N_r\,\mathrm{Sc}_g\,dV_g
\;\ge\;
\int_A \mathrm{Sc}_g\,dV_g.
\end{equation}

Write $k:=n-2$ and identify $(\widetilde M,\widetilde g)=(\Sigma,h)\times(\mathbb R^k,g_E).$ For a Riemannian product, the Ricci tensor splits:
$\mathrm{Ric}_{\widetilde g} = (\mathrm{Ric}_h)\oplus 0.$
Since $\mathrm{Ric}_{\widetilde g}\ge 0$, it follows that $\mathrm{Ric}_h\ge 0$.
In dimension $2$, $\mathrm{Ric}_h = K_h\,h$, where $K_h$ is the Gauss curvature,
so $K_h\ge 0$ and $\mathrm{Sc}_h = 2K_h\ge 0$.
Moreover, scalar curvature is additive under products, hence
\[
\mathrm{Sc}_{\widetilde g} = \mathrm{Sc}_h + \mathrm{Sc}_{g_E}
=\mathrm{Sc}_h,
\qquad
dV_{\widetilde g}=dA_h\,dz.
\]
Since $(\widetilde M,\widetilde g)$ is complete  and a Riemannian
product is complete if and only if  each factor is complete,  $(\Sigma,h)$ is complete. 
 As $\widetilde M$ is simply connected, $\Sigma$ is simply connected. By the classification of simply connected
surfaces, $\Sigma\cong S^2$ if compact and $\Sigma\cong \mathbb R^2$ if noncompact.

If $\Sigma$ is compact, then $\Sigma\cong S^2$ and Gauss--Bonnet gives
\[
\int_\Sigma K_h\,dA_h = 2\pi\,\chi(\Sigma)=2\pi\cdot 2=4\pi,
\]
hence $\int_\Sigma \mathrm{Sc}_h\,dA_h = 2\int_\Sigma K_h\,dA_h = 8\pi$.

If $\Sigma$ is noncompact then $\Sigma\cong \mathbb R^2$ and $K_h\ge 0$. By \cite[Theorem~2.1]{2016arXiv160907631C}, the noncompact surface $(\Sigma,h)$ has finite total curvature. Consequently, the Cohn--Vossen inequality yields
\[
\int_{\Sigma} K_h\,dA_h \le 2\pi\,\chi(\Sigma)=2\pi.
\]
Therefore,
\[
\int_\Sigma \mathrm{Sc}_h\,dA_h = 2\int_\Sigma K_h\,dA_h \le 4\pi.
\]

Using the product identification, write $\widetilde p=(p_0,z_0)\in \Sigma\times\mathbb R^k$.
Composing with the isometry $(x,z)\mapsto (x,z-z_0)$ of the Euclidean factor,
we may assume $\widetilde p=(p_0,0)$. For the product metric, one has the distance formula
\begin{equation*}
d_{\widetilde g}\bigl((x,z),(p_0,0)\bigr)^2
=
d_h(x,p_0)^2 + |z|^2.
\end{equation*}
Consequently,
\[
\widetilde A
=
\bigl\{(x,z): |z|<r,\ d_h(x,p_0)<\sqrt{r^2-|z|^2}\bigr\}.
\]

Since $\mathrm{Sc}_{\widetilde g}=\mathrm{Sc}_h\ge 0$, we may apply Tonelli's
theorem to integrate over the product:
\begin{align*}
\int_{\widetilde A}\mathrm{Sc}_{\widetilde g}\,dV_{\widetilde g}
&=
\int_{|z|<r}\left(\int_{B_h\bigl(p_0,\sqrt{r^2-|z|^2}\bigr)} \mathrm{Sc}_h\,dA_h\right)dz \\
&\le
\int_{|z|<r}\left(\int_{\Sigma}\mathrm{Sc}_h\,dA_h\right)dz \\
&=
\mathrm{Vol}_{\mathbb R^k}(B_{\mathbb R^k}(0,r)) \int_{\Sigma}\mathrm{Sc}_h\,dA_h \\
&=
\omega_k\,r^k \int_{\Sigma}\mathrm{Sc}_h\,dA_h.
\end{align*}
Since $k=n-2$, this gives
\begin{equation}\label{eq:tilde-ball-bound}
\int_{B_{\widetilde g}(\widetilde p,r)}\mathrm{Sc}_{\widetilde g}\,dV_{\widetilde g}
\le 8\pi\,\omega_{n-2}\,r^{n-2}.
\end{equation}

Combine \eqref{eq:pushdown-fixed} with \eqref{eq:tilde-ball-bound} to obtain
\[
\int_{B_g(p,r)} \mathrm{Sc}_g\,dV_g
\le
\int_{B_{\widetilde g}(\widetilde p,r)}\mathrm{Sc}_{\widetilde g}\,dV_{\widetilde g}
\le
8\pi\,\omega_{n-2}\,r^{n-2}.
\]
\end{proof}

The argument in the proof of Proposition~\ref{codim2-splitting-Sc} also applies to the case of codimension~$1$ splitting in the Riemannian universal cover.

\begin{proposition}\label{cyl_scalar_integral}
Let $(M^n,g)$ be a complete Riemannian manifold of dimension $n \ge 3$.
Assume that $M$ is locally isometric to the unit cylinder
\[
\bigl(\mathbb{R} \times S^{n-1},\, g_{\mathrm{cyl}} := dt^2 + g_{st}\bigr).
\]
 For every point $p \in M$ and every $r>0$, the following inequality holds:
\[ 
\int_{B_g(p,r)} \mathrm{Sc}_g \, dV_g \leq
2(n-1)(n-2)\,\omega_{n-1}\, r,
\]
where $\omega_{n-1}:=\mathrm{Vol}_{g_{st}}(S^{n-1})$. 
\end{proposition}

\begin{proof}
Since $n \ge 3$, the cylinder
$(\mathbb{R} \times S^{n-1},\, g_{\mathrm{cyl}})$ is simply connected and complete.
Because a complete Riemannian manifold that is locally isometric to a complete,
simply connected model has its Riemannian universal cover globally isometric to
that model, it follows that there exists a Riemannian covering map
\[
\pi : (\mathbb{R} \times S^{n-1}, g_{\mathrm{cyl}})
\longrightarrow (M,g).
\]
Since $\pi$ is a local isometry, the scalar curvature is preserved.
Therefore,
\[
\mathrm{Sc}_g \equiv \mathrm{Sc}_{g_{\mathrm{cyl}}}
= (n-1)(n-2)
=: S_0
\quad \text{on } M.
\]
Consequently,
\[
\int_{B_g(p,r)} \mathrm{Sc}_g \, dV_g
= S_0 \, \mathrm{Vol}_g\bigl(B_g(p,r)\bigr).
\]

Fix a lift $\tilde p \in \mathbb{R} \times S^{n-1}$ such that $\pi(\tilde p)=p$.
Let
\[
\widetilde A := B_{g_{\mathrm{cyl}}}(\tilde p,r),
\qquad
A := B_g(p,r).
\]
As in the proof of Proposition~\ref{codim2-splitting-Sc}, we have
$\pi(\widetilde A)=A.$ Define the multiplicity function
\[
N_r(q)
:= \#\bigl(\pi^{-1}(q)\cap \widetilde A\bigr),
\qquad q\in A.
\]
Then $N_r(q)\ge 1$ for all $q\in A$, since $A=\pi(\widetilde A)$.
Moreover, $N_r(q)<\infty$ for every $q\in A$, because $\pi^{-1}(q)$ is discrete
and $\overline{\widetilde A}$ is compact by the Hopf--Rinow theorem.

Consequently, using Lemma \ref{coarea}, we obtain
\[
\mathrm{Vol}_{g_{\mathrm{cyl}}}(\widetilde A)
=
\int_A N_r(q)\, dV_g(q)
\;\ge\;
\int_A 1\, dV_g
=
\mathrm{Vol}_g(A).
\]

For any point $\tilde{p} = (t_0, x_0)$ and $r > 0$, the squared distance in the product metric $\tilde{g}$ is given by $$d_{\tilde{g}}\bigl((t, x), (t_0, x_0)\bigr)^2 = (t - t_0)^2 + d_{g_{st}}(x, x_0)^2.$$The condition for a point $(t, x)$ to lie within the geodesic ball $B_{\tilde{g}}(\tilde{p}, r)$ is thus$$(t - t_0)^2 + d_{S^{n-1}}(x, x_0)^2 < r^2.$$Since $d_{g_{st}}(x, x_0)^2 \geq 0$, it follows that $(t - t_0)^2 < r^2$, yielding the containment$$B_{\tilde{g}}(\tilde{p}, r) \subset (t_0 - r, t_0 + r) \times S^{n-1}.$$
Thus,
\[
\mathrm{Vol}_g(B_g(p,r)) \leq \mathrm{Vol}_{g_{\mathrm{cyl}}}\bigl(B_{g_{\mathrm{cyl}}}(\tilde p,r)\bigr)
\le 2r\,\mathrm{Vol}_{g_{st}}(S^{n-1}),
\]
which proves \[ 
\int_{B_g(p,r)} \mathrm{Sc}_g \, dV_g \leq
2(n-1)(n-2)\,\mathrm{Vol}_{g_{st}}(S^{n-1})\, r.
\]
\end{proof}

\begin{proof}[Proof of Theorem \ref{F}]
The assumption in part~(i) implies that the Riemannian universal cover of $(M^n,g)$
splits off an isometric factor $\mathbb{R}^{n-2}$; see \cite[Theorem~1.2]{2026arXiv260114231C}.
Consequently, Proposition~\ref{codim2-splitting-Sc} applies.

Assume that $(M^n,g)$ is LCF. We consider the Zhu–Carron–Herzlich classification case by case. If $(M^n,g)$ is globally conformally equivalent to a space form of positive
sectional curvature, then $M$ is compact, which contradicts the assumption that
$M$ is noncompact.  If $(M^n,g)$ is flat, the conclusion is immediate. If instead $(M^n,g)$ is globally conformal to the flat metric on $\mathbb{R}^n$,
then the conclusion follows from Proposition~\ref{estimate}.
If $(M^n,g)$ is locally isometric to the unit cylinder
$\bigl(\mathbb{R}\times S^{n-1},\, g_{\mathrm{cyl}}:=dt^2+g_{st}\bigr),$
then the desired conclusion follows directly from
Proposition~\ref{cyl_scalar_integral}.
This exhausts all cases in the classification and completes the proof of
part~(ii).
\end{proof}

\section{Cohn--Vossen-Type Inequalities  in Dimension 3}\label{dim 3}
 Gromov~\cite[2.B.]{zbMATH03969577} sketches a proof of the following result: if a Riemannian manifold $(M^n,g)$ has sectional curvature $K_g \ge 0$ and scalar curvature $\mathrm{Sc}_g \ge 1$, then there exists a constant $C(n)>0$ such that $\operatorname{Vol}_g(B_g(p,r)) \le C(n)\, r^{\,n-2}.$
He further conjectures~\cite[2.A.(b)]{zbMATH03969577} that the same conclusion should hold when the assumption of nonnegative sectional curvature is relaxed to nonnegative Ricci curvature. In other words, Gromov’s conjecture states that if a complete connected Riemannian manifold $(M^n,g)$ satisfies $\mathrm{Ric}_g \ge 0$ and $\mathrm{Sc}_g \ge \sigma^2 > 0$, then there exists a constant $C = C(n,\sigma)$ such that
$\sup_{p \in M} \operatorname{Vol}_g\big(B_g(p,r)\big) \le C(n,\sigma)\, r^{\,n-2}$
for all  $r>0$.

In dimension~$3$, by applying the Schoen--Yau formula for minimal surfaces to level sets of harmonic functions, Munteanu and Wang verify Gromov’s conjecture in \cite[Theorem~1.2]{2022arXiv220105595M}.   

By a result of Schoen--Yau--Liu~\cite{zbMATH06210494}, if $(M^3,g)$ is a complete noncompact manifold with nonnegative Ricci curvature, then either $M^3$ is diffeomorphic to $\mathbb{R}^3$, or its Riemannian universal cover $\widetilde{M}^3$ is isometric to a Riemannian product $\Sigma^2 \times \mathbb{R}$, where $\Sigma^2$ is a complete $2$-manifold with nonnegative sectional curvature.

\begin{remark}

Proposition~\ref{codim2-splitting-Sc} establishes Gromov's conjecture both in the latter case of the Schoen--Yau--Liu theorem for $n=3$, and  in the general case for $n\geq 3$ under the assumption that $\pi_1(M)$ contains a subgroup isomorphic to $\mathbb{Z}^{n-2}$. Therefore, when addressing Yau's question and Gromov's conjecture in dimension $3$, it suffices to restrict attention to manifolds diffeomorphic to $\mathbb{R}^3$.
\end{remark}

Set
\[
\mathcal{C}(g):=\limsup_{r \to \infty} \frac{1}{r} \int_{B_g(p,r)} \mathrm{Sc}_g \, dV_g .
\]
This quantity is independent of the choice of the base point \(p\), and hence is well defined. Moreover, it is invariant under scaling of the metric.

Recently,  Xu~\cite{MR4802637} shows that Yau's question admits a positive answer for a $3$-manifold $(M^3,g)$ that admits a pole. On the other hand, Munteanu and Wang~\cite{2025arXiv250510520M} confirm Yau's question under the additional assumption that the scalar curvature $\mathrm{Sc}_g$ is bounded between two positive constants. More precisely, both works show that $\mathcal{C}(g) \le 8\pi.$ 
This leads Munteanu and Wang \cite[Conjecture~1.3]{2025arXiv250510520M} to propose the following conjecture: 

\begin{quote}
\textit{Munteanu--Wang’s conjecture.} 
Let $(M^3,g)$ be a three-dimensional complete Riemannian manifold with $\mathrm{Ric}_g \ge 0$ and exactly one end. Then its scalar curvature satisfies $\mathcal{C}(g)\le 8\pi$.
\end{quote}

Since the product metric \(g_{st} \oplus dt^{2}\) on \(S^{2}\times \mathbb{R}\) is complete and has nonnegative Ricci curvature with \(\mathcal{C}(g)=16\pi\), this example shows that the assumption that \(M^3\) has exactly one end in the conjecture of Munteanu--Wang is necessary.

 Inspired by the result of Xu~\cite{MR4802637}, Chen, Xu, and Zhang~\cite[Conjecture~1.7]{2026arXiv260210393C} proposed a related conjecture in which the assumption of exactly one end is replaced by Euclidean volume growth. They conjectured that $\mathcal{C}(g)= 8\pi\bigl[1-\mathrm{AVR}(g)\bigr].$

Since on a complete \(3\)-manifold \((M^3,g)\) with \(\mathrm{Ric}_g \ge 0\), the condition \(\mathrm{AVR}(g)>0\) implies that \(M\) has exactly one end, it follows that \(M^3\) is diffeomorphic to \(\mathbb{R}^3\). Thus, the conjecture of Chen, Xu, and Zhang can be viewed as a refinement of the conjecture of Munteanu and Wang.

For Zhu’s example $(\mathbb{R}^{3},\, g = (r^{2}+1)^{-2\alpha} g_{E})$, where 
$r^{2} = \sum_{i=1}^{3} x_{i}^{2}$ with $\alpha \in \left[0, \tfrac{1}{2}\right]$,
a direct computation shows that $\mathcal{C}(g)=32\pi \alpha (1-\alpha)\le 8\pi .$
This supports the two conjectures mentioned above. We next give a partial result toward Munteanu-Wang's conjecture and Yau's problem.

\begin{theorem}\label{3D Cohn}
Let $(M^3,g)$ be a complete smooth Riemannian $3$-manifold with $\mathrm{Ric}_g \ge 0$.
\begin{enumerate}
\item[I.] Fix $p\in M$ and define $\rho(x):=d_g(p,x).$
Assume that there exist constants $C>0$ and $R_0>0$ such that
\[
\mathrm{Sc}_g(x)\le \frac{C}{\rho(x)^2}
\qquad \text{for all } x\in M \text{ with } \rho(x)\ge R_0 .
\]
 Then $\mathcal{C}(g)\leq 8\pi.$

\item[II.] Suppose there exists a compact domain with smooth nonempty boundary $\Omega\Subset M^3$ such that $\Omega=\overline{\operatorname{Int}(\Omega)},$
and the set $E:=M^3\setminus \operatorname{Int}\Omega$
is connected. Let $E^\circ:=M^3\setminus \Omega$. Assume further that there exists a proper function $f:E\to[0,\infty)$
such that
\[
f\in C^0(E)\cap C^\infty(E^\circ), \qquad f^{-1}(0)=\partial E, \qquad |\nabla f|_g=1 \ \text{on } E^\circ .
\]
Then $\mathcal{C}(g)$ is finite.

\end{enumerate}

\end{theorem}

The positivity assumption on the scalar curvature in the work of Munteanu and Wang is  essential, since their argument relies on the \(\mu\)-bubble diameter estimate. Nevertheless, their argument can be generalized to situations where \(\mathrm{Sc}_g\) satisfies bounds depending on the distance function, for example
$C_0 \rho(x)^{-a} \le \mathrm{Sc}_g(x) \le C_0 \rho(x)^b,$
where \(a,b>0\) are sufficiently small universal constants. In contrast, our proof of Theorem~\ref{3D Cohn}~(I) does not rely on the \(\mu\)-bubble diameter estimate, and therefore allows  \(\mathrm{Sc}_g\) to vanish. The limitation of this argument is that, in the case $\mathrm{AVR}(g)>0$, it yields only the upper bound $\mathcal C(g)\leq 8\pi(1-\mathrm{AVR}(g)),$
rather than the identity required to confirm the conjecture of Chen, Xu, and Zhang.

A Riemannian manifold is said to admit a pole if there exists a point \(p\in M\) such that the exponential map \(\exp_p:T_pM\to M\) is a diffeomorphism. The existence of a pole is a crucial assumption in Xu's argument. If the manifold admits a point whose cut locus is compact, then the corresponding distance function provides an example of the function \(f\) appearing in Theorem~\ref{3D Cohn}~(II). Thus the condition in (II) is weaker than Xu's assumption, although the conclusion is correspondingly weaker. Instead of relying on Xu’s pole-based radial asymptotic analysis, we employ a leaf-based Jacobian concavity argument along a geodesic foliation.

\subsection{Cohn--Vossen-type estimates under quadratic scalar curvature decay}

We first prove a Cohn--Vossen-type estimate in dimension $3$ under the
borderline decay assumption $\mathrm{Sc}_g=O(\rho^{-2})$. This estimate gives
the sharp upper bound in terms of the asymptotic volume ratio.

\begin{proposition}\label{avr}
Let \((M^3,g)\) be a complete Riemannian \(3\)-manifold with \(\mathrm{Ric}_g \ge 0\). Fix \(p \in M\) and define \(\rho(x) := d_g(p,x)\). Suppose there exist constants \(C>0\) and \(R_0>0\) such that
\[
\mathrm{Sc}_g(x) \le \frac{C}{\rho(x)^2}
\qquad \text{for all } x \in M \text{ with } \rho(x) \ge R_0.
\]
Then  \(\mathcal{C}(g) \le 8\pi \).

\end{proposition}

The scalar curvature decay, together with $\mathrm{Ric}_g\ge0$, gives quadratic
Ricci curvature decay. By Reiris's theorem \cite[Corollary~1.1]{zbMATH06423474}, after the elementary coarea estimate, only the Euclidean-volume-growth case remains. There we use Xu's Green function exhaustion: Xu's  weighted estimate \cite{zbMATH07177335} is converted, via the quadratic decay and a dyadic de-weighting argument, into an unweighted estimate on geodesic balls.

\begin{proof}

Fix $p\in M$ and write $B(r):=B_g(p,r)$ and $V(r):=\mathrm{Vol}_g(B(r))$. By the Bishop--Gromov volume comparison, the ratio $V(r)/(\omega_3 r^3)$ is nonincreasing. Therefore
$$
\nu:=\mathrm{AVR}(g):=\lim_{r\to\infty}\frac{V(r)}{\omega_3 r^3}
$$
exists and belongs to $[0,1]$, where $\omega_3=\frac{4\pi}{3}$.

At each point, diagonalize $\mathrm{Ric}_g$ with respect to an orthonormal basis. Since $\mathrm{Ric}_g \geq 0$, its eigenvalues $\lambda_1,\lambda_2,\lambda_3$ are nonnegative. Hence
$$
\mathrm{Sc}_g=\lambda_1+\lambda_2+\lambda_3,
\qquad
|\mathrm{Ric}_g|=\left(\lambda_1^2+\lambda_2^2+\lambda_3^2\right)^{1/2}.
$$
It follows that $|\mathrm{Ric}_g|\leq \mathrm{Sc}_g \leq \sqrt{3}|\mathrm{Ric}_g|.$
Therefore, under the assumption $\mathrm{Ric}_g\geq 0$, quadratic decay of the scalar curvature is equivalent to quadratic decay of the Ricci curvature.

Since \(M\) is now assumed to be diffeomorphic to \(\mathbb R^3\), 
Reiris's theorem applies. Namely, Reiris proves that a noncompact, 
boundaryless three-dimensional manifold with nonnegative Ricci curvature 
and quadratic Ricci curvature  decay has Euclidean volume growth if and only if 
it is diffeomorphic to \(\mathbb R^3\) 
\cite[Corollary~1.1]{zbMATH06423474}.

If $M$ is not diffeomorphic to $\mathbb{R}^3$, then Reiris’s theorem implies $\nu=0$; we address this case first. Hence, for every $r\ge R_0$,
\[
\int_{B(r)}\mathrm{Sc}_g\,dV_g
\le
\int_{B(R_0)}\mathrm{Sc}_g\,dV_g
+
C\int_{B(r)\setminus B(R_0)} \rho^{-2}\,dV_g.
\]
The distance function $\rho(x) := d_g(p, x)$ is $1$-Lipschitz on $(M, g)$ and, by Rademacher’s theorem, differentiable $dV_g$-a.e. with $|\nabla \rho|_g = 1$. The Coarea formula for Lipschitz maps implies that the volume $V(t) := \operatorname{Vol}_g(B_g(p,t))$ is the primitive of the $2$-dimensional Hausdorff measure of the geodesic spheres, satisfying $V'(t) = \mathcal{H}^2(\rho^{-1}(t))$ for a.e. $t > 0$. Under the assumption $\mathrm{Ric}_g \ge 0$, the Bishop–Gromov monotonicity theorem states that $t \mapsto V(t)/t^3$ is nonincreasing, which implies $V(t) \le \frac{4\pi}{3}t^3$ and $V'(t) \le 4\pi t^2$ for a.e. $t > 0$. Consequently, the sphere-area function $t \mapsto \mathcal{H}^2(\rho^{-1}(t))$ is bounded by $4\pi t^2$ for a.e. $t$ and thus belongs to $L^1_{\mathrm{loc}}(0, \infty)$. This ensures that $V$ is locally absolutely continuous, justifying the radial reduction:
$$\int_{B(r) \setminus B(R_0)} \rho^{-2} \, dV_g = \int_{R_0}^r \left( \int_{\rho^{-1}(t)} t^{-2} \, d\mathcal{H}^2 \right) dt = \int_{R_0}^r t^{-2} V'(t) \, dt.$$

Since $V$ is locally absolutely continuous, integration by parts gives
\[
\int_{R_0}^{r} t^{-2}V'(t)\,dt
=
\frac{V(r)}{r^2}-\frac{V(R_0)}{R_0^2}
+2\int_{R_0}^{r}\frac{V(t)}{t^3}\,dt.
\]
Therefore
\[
\frac1r\int_{B(r)}\mathrm{Sc}_g\,dV_g
\le
\frac1r\int_{B(R_0)}\mathrm{Sc}_g\,dV_g
+\frac{C\,V(r)}{r^3}
-\frac{C\,V(R_0)}{R_0^2\,r}
+\frac{2C}{r}\int_{R_0}^{r}\frac{V(t)}{t^3}\,dt.
\]
Since  $\nu=0$, we have 
\[
f(t):= \frac{V(t)}{t^3}=\omega_3\frac{V(t)}{\omega_3 t^3}\longrightarrow 0
\qquad\text{as }t\to\infty.
\]
 Partitioning the integral at $T$, we have for $r \ge T$:$$\frac{1}{r} \int_{R_0}^{r} f(t) \, dt = \frac{1}{r} \int_{R_0}^{T} f(t) \, dt + \frac{1}{r} \int_{T}^{r} f(t) \, dt \le \frac{1}{r} \int_{R_0}^{T} f(t) \, dt + \frac{r-T}{r}\eta.$$Taking the limit superior as $r \to \infty$ on both sides yields$$\limsup_{r \to \infty} \frac{1}{r} \int_{R_0}^{r} \frac{V(t)}{t^3} \, dt \le \limsup_{r \to \infty} \left( \frac{1}{r} \int_{R_0}^{T} f(t) \, dt + \eta \right) = \eta.$$Since $\eta > 0$ is arbitrary, the $\limsup$ is zero, which implies$$\lim_{r \to \infty} \frac{1}{r} \int_{R_0}^{r} \frac{V(t)}{t^3} \, dt = 0.$$
Therefore
\[  
\limsup_{r\to\infty}\frac1r\int_{B(r)}\mathrm{Sc}_g\,dV_g=0.
\]

It remains to consider the case in which $M$ is diffeomorphic to $\mathbb R^3$. In this case, by the preceding application of Reiris's theorem, we have $\nu>0$.
Thus,  the Bishop--Gromov inequality gives $V(r)\ge \nu\,\omega_3 r^3$ for all $r>0.$ Hence
\[
\int_1^\infty \frac{r}{V(r)}\,dr
\le
\frac{1}{\nu\omega_3}\int_1^\infty r^{-2}\,dr<\infty.
\]
By the criterion of Varopoulos, \(M\) is nonparabolic, that is, it admits a positive Green function. On a nonparabolic Riemannian manifold, there exists a unique minimal positive Green function. Let \(G(p,\cdot)\) denote this minimal positive Green function, and define
\[
b(x) = \bigl[n(n-2)\omega_n \cdot G(p,x)\bigr]^{\frac{1}{2-n}},
\]
where \(\omega_n\) denotes the volume of the unit ball in \(\mathbb{E}^n\). This is Xu’s associated \(b\)-function in  \cite{zbMATH07177335}. Since \(n=3\), we set
\[
\widehat{b} := \nu^{-1} b, \qquad E_r := \{\widehat{b} \le r\}.
\] 
 Let $\lambda = \nu^{-1}$. Then $\widehat{b} = \lambda b$, which implies $|\nabla \widehat{b}| = \lambda |\nabla b|$ and $E_r = \{b \le r/\lambda\}$. Setting $s = r/\lambda$, we have:$$\frac{1}{r} \int_{E_r} \mathrm{Sc}_g \, |\nabla \widehat{b}| \, dV_g = \frac{1}{\lambda s} \int_{\{b \le s\}} \mathrm{Sc}_g \, (\lambda |\nabla b|) \, dV_g = \frac{1}{s} \int_{\{b \le s\}} \mathrm{Sc}_g \, |\nabla b| \, dV_g.$$The scale factor $\lambda$ cancels out, demonstrating that the $\limsup$ depends only on the structure of the level sets and the gradient weight. Consequently, applying the estimate from Xu's theorem  \cite[Theorem~1.2]{zbMATH07177335} to $b$ and passing to the limit in $s$ (as $r \to \infty$) yields:
\begin{equation}\label{Xu}
\limsup_{r \to \infty} \frac{1}{r} \int_{E_r} \mathrm{Sc}_g \, |\nabla \widehat{b}| \, dV_g \le 8\pi(1-\nu).
\end{equation}
Since \(n=3\), by Colding--Minicozzi~\cite[Lemma~2.1]{zbMATH07177335} we have
\[
\lim_{\rho(x)\to\infty}\frac{b(x)}{\rho(x)}=\nu.
\]
Therefore, since \(\widehat{b} := \nu^{-1} b\),
\[
\lim_{\rho(x)\to\infty}\frac{\widehat{b}(x)}{\rho(x)} = 1.
\]
 Hence, for every $\varepsilon\in(0,\tfrac12]$, there exists
$R_\varepsilon\ge R_0$ such that
\begin{equation}\label{eplislon }
(1-\varepsilon)\rho(x)\le \widehat b(x)\le (1+\varepsilon)\rho(x)
\qquad\text{whenever }\rho(x)\ge R_\varepsilon.
\end{equation}

We fix \(\varepsilon \in (0,\tfrac12]\) throughout the argument and let \(\varepsilon \downarrow 0\) in the final step. Choose \(R_\varepsilon\) as above and set
\[
K_\varepsilon := \overline{B_g(p, R_\varepsilon)}.
\]
By Hopf--Rinow, $K_\varepsilon$ is compact. Furthermore,  on $M\setminus K_\varepsilon$,
we have $\rho\ge \frac{\widehat b}{1+\varepsilon}.$
Hence the assumed decay of $\mathrm{Sc}_g$ implies
\begin{equation}\label{eq:Sc-bhat}
\mathrm{Sc}_g\le \frac{C(1+\varepsilon)^2}{\widehat b^2}
\qquad\text{on }M\setminus K_\varepsilon.
\end{equation}
Again $\rho$ is Lipschitz and $|\nabla \rho|=1$ almost everywhere, hence
\begin{equation}\label{a.e}
\bigl|1-|\nabla \widehat b|\bigr|
=
\bigl||\nabla \rho|-|\nabla \widehat b|\bigr|
\le
|\nabla \widehat b-\nabla \rho|
\qquad\text{a.e. on }M.
\end{equation}
By \cite[Lemma~2.3]{zbMATH07177335}, we obtain
\begin{equation}\label{eq:L2}
\frac{1}{\mathrm{Vol}_g(E_r)}
\int_{E_r}|\nabla \widehat b-\nabla \rho|^2\,dV_g\longrightarrow 0
\qquad\text{as }r\to\infty.
\end{equation}

The volume of $E_r$ has cubic growth, and  the $L^1$-error of the gradients vanishes relative to the cubic growth of the manifold. Specifically, we have
\begin{equation}\label{L1}
\limsup_{r \to \infty} \frac{\operatorname{Vol}_g(E_r)}{r^3} < \infty,
\qquad \text{and} \qquad
\lim_{r \to \infty} \frac{1}{r^3} \int_{E_r} |\nabla \widehat{b} - \nabla \rho| \, dV_g = 0.
 \end{equation}
Indeed, for  $\varepsilon = 1/2$, \eqref{eplislon }  implies that there exists a constant $R_* \ge 1$ such that $\widehat{b}(x) \ge \frac{1}{2}\rho(x)$ for all $x \in M$ satisfying $\rho(x) \ge R_*$. Let $r \ge R_*$. To establish the inclusion of $E_r$ in a geodesic ball, consider $x \notin B_g(p, 2r+1)$. It follows that $\rho(x) \ge 2r+1 > R_*$, and by the choice of $R_*$:$$\widehat{b}(x) \ge \frac{1}{2}\rho(x) > \frac{1}{2}(2r) = r.$$By contrapositive, if $\widehat{b}(x) \le r$, then $x \in B_g(p, 2r+1)$. Thus, $E_r \subset B_g(p, 2r+1)$ for all $r \ge R_*$. Applying the Bishop–Gromov inequality for $\operatorname{Ric}_g \ge 0$, we have $\operatorname{Vol}_g(B_g(p, R)) \le \frac{4\pi}{3}R^3$. Consequently:$$\frac{\operatorname{Vol}_g(E_r)}{r^3} \le \frac{4\pi}{3} \frac{(2r+1)^3}{r^3}.$$Taking the limit superior as $r \to \infty$ yields $\limsup_{r \to \infty} r^{-3}\operatorname{Vol}_g(E_r) \le \frac{32\pi}{3} < \infty$.

Let $\delta_r$ be the normalized $L^2$-error from \eqref{eq:L2}, satisfying:$$\int_{E_r} |\nabla \widehat{b} - \nabla \rho|^2 \, dV_g = \delta_r \operatorname{Vol}_g(E_r), \quad \text{where } \lim_{r \to \infty} \delta_r = 0.$$Applying the Cauchy–Schwarz inequality to the $L^1$-error $F(r) := \int_{E_r} |\nabla \widehat{b} - \nabla \rho| \, dV_g$:$$F(r) \le \left( \int_{E_r} 1^2 \, dV_g \right)^{1/2} \left( \int_{E_r} |\nabla \widehat{b} - \nabla \rho|^2 \, dV_g \right)^{1/2} = \sqrt{\delta_r} \operatorname{Vol}_g(E_r).$$Dividing by $r^3$:$$\frac{F(r)}{r^3} \le \sqrt{\delta_r} \left( \frac{\operatorname{Vol}_g(E_r)}{r^3} \right).$$Since the term in parentheses is bounded and $\lim_{r \to \infty} \sqrt{\delta_r} = 0$, we conclude that $\lim_{r \to \infty} r^{-3} F(r) = 0$.

We need the following lemma to finish the proof.
\begin{lemma}\label{lem:dyadic-deweight}
Let \((X,\mu)\) be a measure space, and let \(u : X \to [a,\infty)\) be a measurable function with \(a>0\). Define $H(r) := \mu\bigl(\{u \le r\}\bigr)$  for $r \ge a.$
Assume that \(H(r) < \infty\) for every finite \(r\), and that
\[
\lim_{r \to \infty} \frac{H(r)}{r^3} = 0.
\]
Then
\[
\limsup_{r \to \infty} \frac{1}{r} \int_{\{u \le r\}} u^{-2}\, d\mu = 0.
\]
\end{lemma}

\begin{proof}
Fix $\eta>0$. Choose $T\ge a$ such that $H(t)\le \eta t^3$ for all $t\ge T.$ Let $r\ge 2T$, and choose $N\in\mathbb N$ such that
\[
2^N T<r\le 2^{N+1}T.
\]
Then
\[
\int_{\{u\le r\}} u^{-2}\,d\mu
\le
\int_{\{u\le T\}} u^{-2}\,d\mu
+\sum_{j=0}^{N}\int_{\{2^jT<u\le 2^{j+1}T\}} u^{-2}\,d\mu.
\]
Since $u\ge a$ on $\{u\le T\}$,
\[
\int_{\{u\le T\}} u^{-2}\,d\mu \le a^{-2}H(T)=:C_T<\infty.
\]
For each $j=0,\dots,N$,
\[
\int_{\{2^jT<u\le 2^{j+1}T\}} u^{-2}\,d\mu
\le (2^jT)^{-2}\,H(2^{j+1}T)
\le (2^jT)^{-2}\,\eta(2^{j+1}T)^3
=8\eta\,2^jT.
\]
Hence
\[
\int_{\{u\le r\}} u^{-2}\,d\mu
\le C_T+8\eta T\sum_{j=0}^{N}2^j
\le C_T+8\eta T(2^{N+1}-1).
\]
Because $2^N T<r$, we have $2^{N+1}T<2r$, and therefore
\[
\int_{\{u\le r\}} u^{-2}\,d\mu \le C_T+16\eta r.
\]
Dividing by $r$ and letting $r\to\infty$:
\[
\limsup_{r\to\infty}\frac1r\int_{\{u\le r\}} u^{-2}\,d\mu\le 16\eta.
\]
Since $\eta>0$ is arbitrary, the conclusion follows.
\end{proof}

Define a positive measure
\[
d\mu_\varepsilon
:=
\mathbf 1_{M\setminus K_\varepsilon}\,
|\nabla \widehat b-\nabla \rho|\,dV_g.
\]
Then
\[
H_\varepsilon(r):=\mu_\varepsilon(E_r)
=
\int_{E_r\setminus K_\varepsilon}
|\nabla \widehat b-\nabla \rho|\,dV_g.
\]
We now verify all hypotheses of Lemma~\ref{lem:dyadic-deweight} for
$u=\widehat b$ and $\mu=\mu_\varepsilon$.

First, we establish that the measure $\mu_\varepsilon$ satisfies the asymptotic growth condition required by the lemma. By the definition of $H_\varepsilon(r)$ and the non-negativity of the gradient error term, we observe the following inequality:$$H_\varepsilon(r) = \int_{E_r \setminus K_\varepsilon} |\nabla \widehat{b} - \nabla \rho| \, dV_g \le \int_{E_r} |\nabla \widehat{b} - \nabla \rho| \, dV_g = F(r).$$  The non-negativity of $H_\varepsilon(r)$ and \eqref{L1} imply that $$\lim_{r \to \infty} \frac{H_\varepsilon(r)}{r^3} = 0.$$
Second, by \eqref{eplislon },
\[
\widehat b\ge (1-\varepsilon)\rho\ge (1-\varepsilon)R_\varepsilon=:a_\varepsilon>0
\qquad\text{on }M\setminus K_\varepsilon.
\]
Third, we check that $H_\varepsilon(r)<\infty$ for every finite $r$.
Indeed, if $x\in E_r\setminus K_\varepsilon$, then $\widehat b(x)\le r$, and by
\eqref{eplislon },
\[
(1-\varepsilon)\rho(x)\le \widehat b(x)\le r.
\]
Hence
\[
\rho(x)\le \frac{r}{1-\varepsilon}<\frac{r}{1-\varepsilon}+1,
\]
so
\[
E_r\setminus K_\varepsilon
\subset
B_g\!\left(p,\frac{r}{1-\varepsilon}+1\right)\setminus K_\varepsilon.
\]
Therefore
\[
E_r\setminus K_\varepsilon
\subset
\overline{B_g\!\left(p,\frac{r}{1-\varepsilon}+1\right)}\setminus B_g(p,R_\varepsilon),
\]
and the latter is a compact annulus. Since $p\in K_\varepsilon$, the function
$\widehat b$ is smooth on a neighborhood of this annulus, hence
$|\nabla \widehat b|$ is bounded there; also $|\nabla \rho|=1$ almost
everywhere. Thus
\[
|\nabla \widehat b-\nabla \rho|
\le |\nabla \widehat b|+|\nabla \rho|
\]
is integrable on this compact annulus, and therefore $H_\varepsilon(r)<\infty$.

Having verified the hypotheses of Lemma \ref{lem:dyadic-deweight} for the measure $\mu_\varepsilon$ and the potential $u = \widehat{b}$, the lemma yields the following asymptotic vanishing of the weighted gradient error:$$\limsup_{r \to \infty} \frac{1}{r} \int_{E_r \setminus K_\varepsilon} \widehat{b}^{-2} |\nabla \widehat{b} - \nabla \rho|  dV_g = 0.$$  Combining the scalar curvature decay estimate \eqref{eq:Sc-bhat} with the gradient difference bound \eqref{a.e}, we obtain a pointwise estimate for the integrand on $M \setminus K_\varepsilon$:$$\mathrm{Sc}_g \, \bigl| 1 - |\nabla \widehat{b}| \bigr| \le C(1+\varepsilon)^2 \widehat{b}^{-2} |\nabla \widehat{b} - \nabla \rho|.$$Integrating this inequality over  $E_r \setminus K_\varepsilon$ and normalizing by $r$, we conclude:\begin{equation}\label{eq:outside-error-final}
\limsup_{r \to \infty} \frac{1}{r}
\int_{E_r \setminus K_\varepsilon} \mathrm{Sc}_g \,\bigl| 1 - |\nabla \widehat{b}| \bigr| \, dV_g
\le
C(1+\varepsilon)^2 \lim_{r \to \infty} \frac{1}{r}
\int_{E_r \setminus K_\varepsilon} \widehat{b}^{-2} \, |\nabla \widehat{b} - \nabla \rho| \, dV_g
= 0.
\end{equation}
Now define
\[
C_\varepsilon
:=
\int_{K_\varepsilon}\mathrm{Sc}_g\,(1+|\nabla \widehat b|)\,dV_g.
\]
This is finite. Indeed, \(K_\varepsilon\) is compact and \(\mathrm{Sc}_g\) is smooth. Moreover, \(b\) is smooth on \(M \setminus \{p\}\) by elliptic regularity, and~\cite[(2.1)--(2.2)]{zbMATH07177335} yields
\[
\lim_{\rho(x)\to 0} \frac{b(x)}{\rho(x)} = 1,
\qquad
\lim_{\rho(x)\to 0} |\nabla b|(x) = 1.
\]
Hence \(|\nabla \widehat{b}| = \nu^{-1} |\nabla b|\) is bounded near \(p\). For all sufficiently large \(r\) such that \(K_\varepsilon \subset E_r\), we decompose the global error integral. Substituting the constant \(C_\varepsilon\) and normalizing by \(r\), we obtain:$$\frac{1}{r} \int_{E_r} \mathrm{Sc}_g \bigl| 1 - |\nabla \widehat{b}| \bigr| \, dV_g \le \frac{C_\varepsilon}{r} + \frac{1}{r} \int_{E_r \setminus K_\varepsilon} \mathrm{Sc}_g \bigl| 1 - |\nabla \widehat{b}| \bigr| \, dV_g.$$Since $C_\varepsilon$ is a constant independent of $r$, the first term on the right-hand side vanishes as $r \to \infty$. Combining this with the previously established exterior limit in \eqref{eq:outside-error-final}, it follows that:$$\limsup_{r \to \infty} \frac{1}{r} \int_{E_r} \mathrm{Sc}_g \bigl| 1 - |\nabla \widehat{b}| \bigr|dV_g = 0.$$

Since $\mathrm{Ric}_g \ge 0$,  $\mathrm{Sc}_g\geq 0$. By the triangle inequality, the unit constant satisfies the pointwise bound $1 \le |\nabla \widehat{b}| + \bigl| 1 - |\nabla \widehat{b}| \bigr|$. Multiplying by $\mathrm{Sc}_g$ and integrating over  $E_r$, we obtain:$$\int_{E_r} \mathrm{Sc}_g \, dV_g \le \int_{E_r} \mathrm{Sc}_g |\nabla \widehat{b}| \, dV_g + \int_{E_r} \mathrm{Sc}_g \bigl| 1 - |\nabla \widehat{b}| \bigr| \, dV_g.$$Normalizing by $r$, we have:$$\frac{1}{r} \int_{E_r} \mathrm{Sc}_g \, dV_g \le \frac{1}{r} \int_{E_r} \mathrm{Sc}_g |\nabla \widehat{b}| \, dV_g + \frac{1}{r} \int_{E_r} \mathrm{Sc}_g \bigl| 1 - |\nabla \widehat{b}| \bigr| \, dV_g.$$
This implies 
\begin{equation}\label{eq:deweighted-lim}\limsup_{r \to \infty} \frac{1}{r} \int_{E_r} \mathrm{Sc}_g  dV_g \le \limsup_{r \to \infty} \frac{1}{r} \int_{E_r} \mathrm{Sc}_g |\nabla \widehat{b}| dV_g.\end{equation}

Next, \eqref{eplislon } implies that if $x\in B(r)\setminus K_\varepsilon$,
then
\[
\widehat b(x)\le (1+\varepsilon)\rho(x)<(1+\varepsilon)r.
\]
Therefore, for all $r\ge R_\varepsilon$,
\[
B(r)\subset K_\varepsilon\cup E_{(1+\varepsilon)r}.
\]
Then $\mathrm{Sc}_g\ge 0$ yields
\[
\frac{1}{r}\int_{B(r)}\mathrm{Sc}_g\,dV_g
\le
\frac{1}{r}\int_{K_\varepsilon}\mathrm{Sc}_g\,dV_g
+\frac{1}{r}\int_{E_{(1+\varepsilon)r}}\mathrm{Sc}_g\,dV_g.
\]

Using \eqref{eq:deweighted-lim} and \eqref{Xu}, we get
\[
\limsup_{r\to\infty}\frac1r\int_{B(r)}\mathrm{Sc}_g\,dV_g
\le
(1+\varepsilon)
\limsup_{r\to\infty}\frac{1}{(1+\varepsilon)r}\int_{E_{(1+\varepsilon)r}}\mathrm{Sc}_g\,|\nabla \widehat b|\,dV_g
\le
(1+\varepsilon)\,8\pi(1-\nu).
\]
Since this inequality holds for every $\varepsilon > 0$, we may take the limit as $\varepsilon \to 0^+$ to conclude:
\[
\limsup_{r\to\infty}\frac1r\int_{B(r)}\mathrm{Sc}_g\,dV_g
\le
8\pi(1-\nu)\le 8\pi.
\]
\end{proof}

\begin{remark}
More generally, suppose that for some $a,C,R_0>0$,
\[
\mathrm{Sc}_g(x)\le C\rho(x)^{-a}
\qquad\text{whenever }\rho(x)\ge R_0 .
\]
Then $a=2$ is the borderline exponent for the present method: one has
$\mathcal C(g)=0$ when $a>2$, whereas for $0<a<2$ the argument gives no finite
upper bound.

Indeed, since $\mathrm{Ric}_g\ge0$, we have $\mathrm{Sc}_g\ge0$, and the
Bishop--Gromov inequality  gives $V'(t)\le4\pi t^2$ for a.e. $t>0$,
thus the coarea formula yields
\begin{equation*}
\int_{B_g(p,r)\setminus B_g(p,R_0)}\rho^{-a}\,dV_g
=
\int_{R_0}^{r}t^{-a}V'(t)\,dt
\le
4\pi\int_{R_0}^{r}t^{2-a}\,dt .
\end{equation*}
If $a>2$, the right-hand side is $o(r)$, and hence $\mathcal C(g)=0$.

The same threshold appears in the $b$-function de-weighting step. On
$M\setminus K_\varepsilon$, the estimate \eqref{eq:Sc-bhat} is replaced by
$\mathrm{Sc}_g\le C(1+\varepsilon)^a\widehat b^{-a}.$
With
\[
d\mu_\varepsilon
=
\mathbf 1_{M\setminus K_\varepsilon}
|\nabla\widehat b-\nabla\rho|\,dV_g,
\qquad
H_\varepsilon(r):=\mu_\varepsilon(E_r)=o(r^3),
\]
the same dyadic argument as in Lemma~\ref{lem:dyadic-deweight}, with exponent
$a$, gives
\begin{equation*}
\int_{E_r\setminus K_\varepsilon}
\widehat b^{-a}\,|\nabla\widehat b-\nabla\rho|\,dV_g
=
\begin{cases}
o(r^{3-a}), & 0<a<3,\\[0.3ex]
o(\log r), & a=3,\\[0.3ex]
O(1), & a>3.
\end{cases}
\end{equation*}
Thus this error is $o(r)$ exactly when $a>2$.

 These estimates do not control
$\mathcal C(g)$ when $0<a<2$. Moreover, in this range the decay assumption
$\mathrm{Sc}_g=O(\rho^{-a})$ no longer implies quadratic Ricci decay, so the
Reiris input used to obtain Euclidean volume growth is unavailable from this
hypothesis alone. Hence the present method treats the borderline case $a=2$,
gives the stronger conclusion $\mathcal C(g)=0$ for $a>2$, and does not extend
to $0<a<2$.
\end{remark}

\begin{remark}
The argument above gives only the upper bound
$\mathcal C(g)\le 8\pi(1-\nu),$ but it does not imply the identity
$\mathcal C(g)= 8\pi(1-\nu).$
In particular, it does not verify the conjecture of Chen, Xu, and Zhang
\cite[Conjecture~1.7]{2026arXiv260210393C} that for Euclidean volume growth one should have $\mathcal C(g)=8\pi\bigl[1-\mathrm{AVR}(g)\bigr].$
The reason is that the proof  uses only Xu's weighted \textit{upper} bound for
$\int_{E_r}\mathrm{Sc}_g\,|\nabla \widehat b|\,dV_g,$
together with a one-sided de-weighting estimate, and therefore produces only an
upper bound for the geodesic-ball quantity. It does not provide the matching
lower bound, nor does it identify the exact asymptotic behavior of the
unweighted integral.
\end{remark}

\subsection{Cohn--Vossen-type estimates from distance-like exhaustions}

We next treat the remaining three-dimensional case by assuming the existence
of a proper distance-like exhaustion on the exterior region. Its level sets
give a global foliation at infinity, allowing the scalar-curvature integral to
be controlled by the Riccati equation and Gauss--Bonnet.

\begin{proposition}\label{level set}
Let $(M^3,g)=(\mathbb{R}^3,g)$ be a complete smooth Riemannian $3$-manifold with $\mathrm{Ric}_g \ge 0$. Suppose there exists a compact domain with smooth nonempty boundary $\Omega\Subset M^3$ such that $\Omega=\overline{\operatorname{Int}\Omega},$
and the set $E:=M^3\setminus \operatorname{Int}\Omega$
is connected. Let $E^\circ:=M^3\setminus \Omega$. Assume further that there exists a proper function $f:E\to[0,\infty)$
such that
\[
f\in C^0(E)\cap C^\infty(E^\circ), \qquad f^{-1}(0)=\partial E, \qquad |\nabla f|_g=1 \ \text{on } E^\circ .
\]
Then $\mathcal{C}(g)$ is finite.
\end{proposition}

  The proof proceeds by foliating the exterior region of $(M,g)$ into equidistant compact hypersurfaces evolving along a geodesic gradient flow. Under the curvature condition $\mathrm{Ric}_g \ge 0$, the Riccati equation implies that the area Jacobian is  concave, which provides uniform integral control over the extrinsic geometry of the leaves. Substituting these estimates into the Gauss equation and invoking the Gauss–Bonnet theorem, which yields a constant $8\pi$ contribution due to the topological constraint $M \cong \mathbb{R}^3$, we bound the integrated scalar curvature over large cylindrical shells. Finally, by placing arbitrary metric balls inside controlled shells, the claimed asymptotic bound is recovered.

\begin{proof}
We adopt the curvature convention
\[
R(X,Y)Z:=\nabla_X\nabla_YZ-\nabla_Y\nabla_XZ-\nabla_{[X,Y]}Z
\]
throughout. 
Because $\Omega=\overline{\operatorname{Int}\Omega}$, the set $E=M\setminus\operatorname{Int}\Omega$ is closed in $M$, the set $E^\circ=M\setminus\Omega$ is open in $M$, and $\partial E=\partial\Omega\neq\varnothing$.

We fix once and for all an arbitrary auxiliary level $a>0$. All constants appearing below will depend on $a$ but will be finite and independent of the radius $r$ in the final estimate. We divide the proof into seven steps.

\textit{Step 1.} We show that the positive level sets of \( f \) are nonempty, compact, smooth surfaces.
Since $f:E\to[0,\infty)$ is continuous and $E$ is connected, the image $f(E)$ is a connected subset of $[0,\infty)$. Moreover, $0\in f(E)$ because $f^{-1}(0)=\partial E\neq\varnothing$.

 We claim that $f(E)$ is unbounded above. Suppose for contradiction that $f(E) \subset [0, A]$ for some $A < \infty$. Then $E = f^{-1}([0, A])$. Because $f: E \to [0, \infty)$ is proper, the preimage of the compact interval $[0, A]$ must be compact. This would imply that $E$ itself is a compact subset of $M$. But $M$ is diffeomorphic to $\mathbb{R}^3$ and $\Omega$ is compact, so $E=M\setminus\operatorname{Int}\Omega$ is unbounded and  noncompact, a contradiction. Therefore $f(E)$ is unbounded above. Since $f(E)\subset[0,\infty)$ is connected and contains $0$, it follows that $f(E)=[0,\infty)$.

Now fix $t>0$ and set $\Sigma_t:=f^{-1}(t)$. Then $\Sigma_t\neq\varnothing$. If $x\in\Sigma_t$, then $f(x)=t>0$, so $x\notin f^{-1}(0)=\partial E$ and hence $x\in E^\circ$. On $E^\circ$ we have $|\nabla f|_g=1$, so $df\neq 0$ at every point of $\Sigma_t$. Therefore $\Sigma_t$ is a smooth embedded hypersurface in $E^\circ\subset M$. Since $\{t\}$ is compact in $[0,\infty)$ and $f$ is proper as a map from $E$ to $[0,\infty)$, the preimage $\Sigma_t=f^{-1}(\{t\})$ is compact in $E$ and hence compact in $M$. Thus, for every $t>0$, $\Sigma_t\subset E^\circ$
is a smooth compact embedded surface in $M$.

\textit{Step 2.} We show that the gradient flow is geodesic.
Define the unit normal vector field
\[
\nu:=\nabla f\qquad\text{on }E^\circ.
\]
We claim that
\[
\nabla_\nu\nu=0\qquad\text{on }E^\circ.
\]
Indeed, let $X$ be an arbitrary smooth vector field on $E^\circ$. Since $|\nu|_g^2=1$,
\[
0=\frac12X(|\nu|_g^2)=\langle\nabla_X\nu,\nu\rangle_g.
\]
Because $\nu=\nabla f$, the right-hand side equals the Hessian
$\nabla^2f(X,\nu).$ Since the Hessian of a smooth function is symmetric, we have 
\[
0=\nabla^2f(X,\nu)=\nabla^2f(\nu,X)=\langle\nabla_\nu\nu,X\rangle_g.
\]
Hence
\[
\langle\nabla_\nu\nu,X\rangle_g=0\qquad\text{for every smooth vector field }X\text{ on }E^\circ.
\]
Since the tangent space at each point is spanned by arbitrary vectors, it follows that $\nabla_\nu\nu=0$ on $E^\circ$. Consequently, every integral curve of \( \nu \) is a unit-speed geodesic, as follows from the geodesic equation.

\textit{Step 3.} We show that  the superlevel set $U_a$ is connected.
Define
\[
U_a:=f^{-1}((a,\infty)),\qquad E_a:=f^{-1}([a,\infty)),\qquad C_a:=\Omega\cup f^{-1}([0,a]).
\]
Since $[0,a]$ is compact in $[0,\infty)$ and $f:E\to[0,\infty)$ is proper, the preimage $f^{-1}([0,a])$ is compact in $E$ and hence compact in $M$. Therefore $C_a$ is compact in $M$, and $U_a=M\setminus C_a.$

We first show that no connected component of $U_a$ is  relatively compact in $M$. Let $x\in U_a$ be arbitrary, and let $\gamma:[0,b)\to E^\circ$
be the maximal integral curve of the vector field $\nu$ with $\gamma(0)=x$. Since $\gamma'=\nu$ and $|\nabla f|_g=1$ on $E^\circ$,
\[
\frac{d}{ds}f(\gamma(s))=df_{\gamma(s)}(\gamma'(s))=df_{\gamma(s)}(\nu)=|\nabla f|_g^2(\gamma(s))=1
\]
for all $s\in[0,b)$. Integrating from $0$ gives
\[
f(\gamma(s))=f(x)+s\qquad\text{for all }s\in[0,b).
\]
We claim that $b=\infty$. Suppose for contradiction that $b<\infty$. Since $|\gamma'|_g=1$, for all $0\le s<t<b$ we have
\[
d_g(\gamma(s),\gamma(t))\le\int_s^t|\gamma'(\tau)|_g\,d\tau=t-s.
\]
Thus $\{\gamma(s)\}_{s\in[0,b)}$ is Cauchy as $s\uparrow b$. Because $(M,g)$ is complete, the Hopf--Rinow theorem implies that $(M,d_g)$ is a complete metric space, so there exists $y\in M$ such that $\gamma(s)\to y$ as $s\uparrow b$. Since $E$ is closed in $M$ and $\gamma([0,b))\subset E$, we have $y\in E$. By continuity of $f$ on $E$,
\[
f(y)=\lim_{s\uparrow b}f(\gamma(s))=f(x)+b>a>0,
\]
so $y\notin\partial E$ and hence $y\in E^\circ$. The vector field $\nu$ is smooth on the open set $E^\circ$. By the standard continuation theorem for ordinary differential equations on manifolds, the integral curve $\gamma$ extends to an interval strictly larger than $[0,b)$, contradicting maximality of $b$. Therefore $b=\infty$.

Since $f(\gamma(s))=f(x)+s\to\infty$ as $s\to\infty$, the image $\gamma([0,\infty))$ cannot be relatively compact in $M$. Indeed, if its closure \( \overline{\gamma([0,\infty))} \) were compact, then, since \( E \) is closed and \( \gamma([0,\infty)) \subset E \), this closure would be a compact subset of \( E \). 
However, the continuous function \( f \) would then be bounded on this set, contradicting the fact that \( f(\gamma(s)) \to \infty \). The curve $\gamma([0,\infty))$ is connected, lies entirely in $U_a$, and contains $x$. Therefore the connected component of $x$ in $U_a$ is not relatively compact in $M$. Since $x\in U_a$ was arbitrary, every connected component of $U_a$ is not relatively compact in $M$.

We now use the topology of $M$. Since $M$ is diffeomorphic to $\mathbb{R}^3$, the complement of any compact subset of $M$ has at most one non relatively compact component. Indeed, let $\Psi:M\to\mathbb{R}^3$ be a fixed diffeomorphism. Let $K\subset M$ be any compact set. Then \(\Psi(K)\) is compact in \(\mathbb{R}^3\). Hence there exists \(R_0 > 0\) such that \(\Psi(K)\) is contained in the Euclidean \(R_0\)-ball \(\overline{B}_{\mathbb{E}^3}(0,R_0)\).
Let $V$ be any connected component of $\mathbb{R}^3\setminus\Psi(K)$ that is not relatively compact in $\mathbb{R}^3$. Then $V$ must intersect the exterior region $\mathbb{R}^3\setminus\overline{B}_{\mathbb{E}^3}(0,R_0)$; otherwise $V\subset\overline{B}_{\mathbb{E}^3}(0,R_0)$ and $V$ would be relatively compact. The set $\mathbb{R}^3\setminus\overline{B}_{\mathbb{E}^3}(0,R_0)$ is connected and contained in $\mathbb{R}^3\setminus\Psi(K)$, so it must lie entirely in the component $V$. Therefore every non-relatively-compact connected component of $\mathbb{R}^3\setminus\Psi(K)$ contains this same unbounded exterior region and hence coincides with it. In particular, there is at most one such component. Pulling back by the diffeomorphism $\Psi$, the same conclusion holds for complements in $M$.

Applying this topological fact to the compact set $C_a$, the open set $U_a=M\setminus C_a$ has at most one connected component that is not relatively compact in $M$. But we have already shown that every connected component of $U_a$ is not relatively compact. Therefore $U_a$ has exactly one connected component, i.e., $U_a\text{ is connected.}$

\textit{Step 4. } We establish a product decomposition on the open cylinder.
Let $\Phi_s$ denote the local flow of the vector field $\nu$ on $E^\circ$. For each $y\in\Sigma_a$, let
\[
\gamma_y:[0,b_y)\to E^\circ
\]
be the maximal integral curve of $\nu$ with $\gamma_y(0)=y$. By the same completeness argument used above, the completeness of $(M, g)$ and the fact that $f \circ \gamma_y(s) = a + s$ ensure that the integral curves starting at any $y \in \Sigma_a$ are defined for all $s \in [0, \infty)$. Define
\[
\overline{F}:\Sigma_a\times[0,\infty)\to E_a,\qquad\overline{F}(y,s):=\gamma_y(s),
\]
and let
\[
F:=\overline{F}|_{\Sigma_a\times(0,\infty)}:\Sigma_a\times(0,\infty)\to U_a.
\]
We first verify that $\overline{F}$ (and hence $F$) is smooth. Fix an arbitrary $\Lambda>0$. For every $(y,s)\in\Sigma_a\times[0,\Lambda]$ we have
\[
f(\overline{F}(y,s))=a+s\in[a,a+\Lambda],
\]
so
\[
\overline{F}(\Sigma_a\times[0,\Lambda])\subset E_{a,a+\Lambda}:=f^{-1}([a,a+\Lambda]).
\]
Since $[a,a+\Lambda]$ is compact in $[0,\infty)$ and $f$ is proper, $E_{a,a+\Lambda}$ is compact in $E$ and hence compact in $M$. Because $a>0$, we have $E_{a,a+\Lambda}\subset E^\circ$. Thus $\nu$ is smooth on an open neighborhood of the compact set $E_{a,a+\Lambda}$. By the standard smooth dependence of solutions of ODEs on initial data and time, the map
\[
\overline{F}:\Sigma_a\times[0,\Lambda]\to E^\circ
\]
is smooth. Since $\Lambda>0$ was arbitrary, $\overline{F}$ is smooth on $\Sigma_a\times[0,\infty)$, and therefore $F$ is smooth on $\Sigma_a\times(0,\infty)$.

For every $(y,s)\in\Sigma_a\times(0,\infty)$,
\[
\frac{d}{ds}f(F(y,s))=df_{F(y,s)}(\partial_sF)=df_{F(y,s)}(\nu)=|\nabla f|_g^2(F(y,s))=1.
\]
Integrating from $0$ and using $f(y)=a$  yields
\[
f(F(y,s))=a+s\qquad\text{for all }(y,s)\in\Sigma_a\times(0,\infty).
\]

Furthermore, the map $F:\Sigma_a\times(0,\infty)\to U_a$ is a diffeomorphism.
 Suppose $F(y_1,s_1)=F(y_2,s_2)$. Applying $f$ gives
\[
a+s_1=f(F(y_1,s_1))=f(F(y_2,s_2))=a+s_2,
\]
so $s_1=s_2$.  Since $\nu$ is a smooth vector field on the open set $E^\circ$, it is locally Lipschitz. By the Picard--Lindelöf theorem, the integral curves of $\nu$ are unique. This implies that $y_1=y_2$, and therefore $F$ is injective.

Now we prove the surjectivity of \(F\). Let $x\in U_a$ be arbitrary and set $L:=f(x)-a>0$. Let $\alpha:[0,b)\to E^\circ$
be the maximal integral curve of $-\nu$ with $\alpha(0)=x$. Then
\[
\frac{d}{dt}f(\alpha(t))=df_{\alpha(t)}(\alpha'(t))=-df_{\alpha(t)}(\nu)=-1,
\]
so
\[
f(\alpha(t))=f(x)-t\qquad\text{for all }t\in[0,b).
\]
We claim that $b>L$. Suppose instead that $b\le L$. Since $|\alpha'|_g=1$, the same Cauchy estimate as above shows that $\alpha(t)\to z\in M$ as $t\uparrow b$ for some $z\in E$ (by closedness of $E$). Continuity of $f$ gives
\[
f(z)=f(x)-b\ge f(x)-L=a>0,
\]
so $z\in E^\circ$. Smoothness of $-\nu$ on $E^\circ$ allows extension of $\alpha$ past $b$, contradicting maximality. Hence $b>L$ and $\alpha(L)$ is well-defined. Set
$y:=\alpha(L).$ Then $f(y)=f(x)-L=a$, so $y\in\Sigma_a$. Define
$\beta(s):=\alpha(L-s)$ for $ 0\le s\le L.$
Then $\beta'(s)=\nu(\beta(s))$ and $\beta(0)=y$, so uniqueness of integral curves gives $\beta(s)=\gamma_y(s)$ for $0\le s\le L$. Evaluating at $s=L$ yields
$x=\beta(L)=\gamma_y(L)=F(y,L).$ Thus $F$ is surjective.

 Next, we prove that \(F\) is a local diffeomorphism. Fix $(y,s) \in \Sigma_a \times (0, \infty)$. To verify that $F$ is a local diffeomorphism, we show that its differential $dF_{(y,s)}$ is a linear isomorphism from $T_y\Sigma_a \oplus \mathbb{R}\partial_s$ to $T_{F(y,s)}M$. The restriction of $dF_{(y,s)}$ to the spatial subspace $T_y\Sigma_a$ coincides with the differential of the flow map $d(\Phi_s)_y$. Differentiating the level-set identity $f(F(y,s)) = a+s$ with respect to $y$ yields $df_{F(y,s)}(dF_{(y,s)}(v)) = 0$ for any $v \in T_y\Sigma_a$, confirming that $dF_{(y,s)}(T_y\Sigma_a) \subseteq T_{F(y,s)}\Sigma_{a+s}$. Since the flow of a smooth vector field is a local diffeomorphism, $d(\Phi_s)_y$ is an isomorphism; its restriction to the 2-dimensional subspace $T_y\Sigma_a$ is thus injective and, by dimension counting, maps onto the tangent plane $T_{F(y,s)}\Sigma_{a+s}$. In the temporal direction, $dF_{(y,s)}(\partial_s) = \nu_{F(y,s)}$, and the relation $df_{F(y,s)}(\nu) = |\nabla f|_g^2 = 1$ ensures that $\nu$ is transversal to the level set. Consequently, $dF_{(y,s)}$ maps the basis of the product space to a linearly independent set in $T_{F(y,s)}M$. Because both domain and codomain are 3-dimensional, $dF_{(y,s)}$ is an isomorphism, and $F$ is a local diffeomorphism by the inverse function theorem.

A bijective local diffeomorphism between smooth manifolds of the same dimension is a diffeomorphism. Therefore
$F:\Sigma_a\times(0,\infty)\xrightarrow{\cong}U_a$ is a diffeomorphism.

For each fixed $s>0$ define $F_s:=F(\,\cdot\,,s):\Sigma_a\to\Sigma_{a+s}$.  This map is injective and smooth, and its differential is the restriction of $dF$ to $T\Sigma_a$, which we already showed is an isomorphism onto $T\Sigma_{a+s}$. To prove surjectivity, let $x\in\Sigma_{a+s}$. Since $x\in U_a$, there are unique $y\in\Sigma_a$ and $\sigma>0$ such that $x=F(y,\sigma)$. Applying $f$ gives $a+\sigma=a+s$, so $\sigma=s$ and $x=F_s(y)$. Thus $F_s$ is bijective and locally diffeomorphic, hence a diffeomorphism.

Since \(U_a\) is connected as shown in Step~3 and \(F\) is a diffeomorphism, the product \(\Sigma_a \times (0,\infty)\) is connected. Hence \(\Sigma_a\) is connected. Because $F_s$ is a diffeomorphism for every $s>0$, each $\Sigma_t$ with $t>a$ is connected. Moreover, each $\Sigma_t$ ($t\ge a$) admits the global unit normal field $\nu=\nabla f$, so every such surface is orientable.

Finally, we identify the closure of \(U_a\) in \(M\) as \(\overline{U_a}^{\,M} = E_a\). Indeed, the inclusion \(\overline{U_a}^{\,M} \subseteq E_a\) follows from the closedness of \(E\) and the continuity of \(f\). Since \(U_a = \{x \in E : f(x) > a\}\), any limit point \(x\) must satisfy \(f(x) \ge a\), and hence \(x \in E_a\). Conversely, for any $x \in E_a$, we either have $f(x) > a$, in which case $x \in U_a$ trivially, or $f(x) = a$, implying $x \in \Sigma_a$. In the latter case, since $\Sigma_a \subset E^\circ$ and $\nu$ is smooth, the local flow $\Phi_s(x)$ is well-defined for sufficiently small $s > 0$. The identity $f(\Phi_s(x)) = a + s > a$ ensures that $\Phi_s(x) \in U_a$ for all $s \in (0, \epsilon)$. As $\Phi_s(x) \to x$ when $s \downarrow 0$, $x$ is a limit point of $U_a$. Thus $E_a \subseteq \overline{U_a}^M$, completing the identification 
\begin{equation}\label{closer}
\overline{U_a}^M = E_a.
\end{equation}

\textit{Step 5.} We compute the Jacobian and the mean curvature evolution on finite strips.
For $t\ge a$ define the shape operator, mean curvature, and auxiliary quantity by
\[
A_t(X):=\nabla_X\nu,\qquad H_t:=\operatorname{tr}(A_t),\qquad G_t:=H_t^2-|A_t|^2.
\]
For $s\ge 0$ write
\[
\overline{F}_s:=\overline{F}(\,\cdot\,,s):\Sigma_a\to\Sigma_{a+s}.
\]
Let $dA_t$ denote the area measure induced by $g$ on $\Sigma_t$. Define the Jacobian function
\[
J:\Sigma_a\times[0,\infty)\to(0,\infty)
\]
by
\[
\overline{F}_s^*(dA_{a+s})=J(\cdot,s)\,dA_a.
\]
Since $\overline{F}_0=\operatorname{id}_{\Sigma_a}$, we have $J(\cdot,0)=1$. For every $s>0$ the map $F_s$ is a diffeomorphism, so $J(\cdot,s)>0$. Set
\[
u:=\sqrt{J}\qquad\text{on }\Sigma_a\times[0,\infty).
\]
Since $\overline{F}_s$ is a diffeomorphism for all $s \ge 0$, the Jacobian $J$ is strictly positive. Consequently, $u = \sqrt{J}$ is a smooth function on the compact strip $\Sigma_a \times [0, T-a]$, and the subsequent differential identities involving $u$ are well-defined.
Denote by $A,H,G$ the pullbacks of $A_{a+s},H_{a+s},G_{a+s}$ via $\overline{F}$. Fix an arbitrary $T>a$. Then the restrictions of $\overline{F}$, $J$, $u$, $A$, $H$, and $G$ to the compact strip $\Sigma_a\times[0,T-a]$ are smooth, and all computations below hold on this strip.

Let $D$ denote the pullback connection on $\overline{F}^*TM$ over $\Sigma_a\times[0,T-a]$ induced by the Levi-Civita connection of $g$. Choose a coordinate patch $U\subset\Sigma_a$ with local coordinates $(x^1,x^2)$. On $U\times[0,T-a]$ define sections of $\overline{F}^*TM$ by
\[
X_i:=d\overline{F}(\partial_{x^i}),\qquad\nu:=d\overline{F}(\partial_s).
\]
Because each $\gamma_y$ is an integral curve of $\nabla f$, the section $\nu$ is precisely the pullback of the ambient vector field $\nabla f$ along $\overline{F}$. Since the Levi-Civita connection on $(M, g)$ is torsion-free, the pullback connection $D$ satisfies the identity $D_U(d\overline{F}(V)) - D_V(d\overline{F}(U)) = d\overline{F}([U, V])$ for any vector fields $U, V$ on the domain. Given that $\partial_s$ and $\partial_{x^i}$ commute, we obtain $D_{\partial_s}X_i = D_{\partial_{x^i}}\nu$.
Let
\[
\gamma_{ij}:=\langle X_i,X_j\rangle_g,\qquad h_{ij}:=\langle D_{\partial_{x^i}}\nu,X_j\rangle_g.
\]
Then $(\gamma_{ij})$ is the pullback of the induced metric on $\Sigma_{a+s}$ and $(h_{ij})$ is the pullback of the second fundamental form; moreover $H=\gamma^{ij}h_{ij}$. Since $\nu = \nabla f$, the components of the second fundamental form are given by: $h_{ij} = \nabla^2 f(\partial_{x^i}, \partial_{x^j}).$ Because the Hessian is a symmetric $(0, 2)$-tensor, we have $h_{ij} = h_{ji}$. In terms of the inner products, this implies:$$\langle D_{\partial_{x^i}} \nu, X_j \rangle_g = \langle X_i, D_{\partial_{x^j}} \nu \rangle_g.$$

The evolution of the metric coefficients is
\[
\partial_s\gamma_{ij}=\langle D_{\partial_s}X_i,X_j\rangle_g+\langle X_i,D_{\partial_s}X_j\rangle_g=\langle D_{\partial_{x^i}}\nu,X_j\rangle_g+\langle X_i,D_{\partial_{x^j}}\nu\rangle_g=2h_{ij}.
\]
The Jacobian satisfies
\[
J=\sqrt{\frac{\det(\gamma_{ij}(\cdot,s))}{\det(\gamma_{ij}(\cdot,0))}},
\]
so
$$\partial_s (\log J) = \frac{1}{2} \frac{\partial_s \det(\gamma_{ij}(\cdot,s))}{\det(\gamma_{ij}(\cdot,s))} = \frac{1}{2} \frac{\det(\gamma_{ij}(\cdot,s)) \cdot \gamma^{ij} \partial_s \gamma_{ij}}{\det(\gamma_{ij}(\cdot,s))} = \frac{1}{2} \gamma^{ij} \partial_s \gamma_{ij}=\gamma^{ij}h_{ij}=H,$$
and therefore
\begin{equation}\label{partial J}
\partial_sJ=HJ,\qquad\partial_su=  \frac{1}{2} J^{-1/2} \cdot \partial_s J=\frac12Hu.
\end{equation}
Next we compute the evolution of $H$. Differentiating the second fundamental form,
\[
\partial_sh_{ij}=\langle D_{\partial_s}D_{\partial_{x^i}}\nu,X_j\rangle_g+\langle D_{\partial_{x^i}}\nu,D_{\partial_s}X_j\rangle_g.
\]
Since $D_{\partial_s} \nu = \nabla_{d\overline{F}(\partial_s)} \nu = \nabla_\nu \nu=0$  and $[\partial_s,\partial_{x^i}]=0$, the curvature identity for the pullback connection reads
\[
D_{\partial_s}D_{\partial_{x^i}}\nu-D_{\partial_{x^i}}D_{\partial_s}\nu=R(\nu,X_i)\nu,
\]
so
\[
D_{\partial_s}D_{\partial_{x^i}}\nu=R(\nu,X_i)\nu.
\]
Moreover,
\[
D_{\partial_{x^i}} \nu
= \nabla_{d\overline{F}(\partial_{x^i})}\nu
= \nabla_{X_i}\nu
= A(X_i),
\qquad
D_{\partial_s} X_j
= D_{\partial_{x^j}}\nu
= A(X_j).
\]

Since
\[
\partial_s(\gamma^{ij}\gamma_{jk})
=(\partial_s\gamma^{ij})\gamma_{jk}
+\gamma^{ij}(\partial_s\gamma_{jk})
=\partial_s(\delta^i_k)=0,
\]
and
$\partial_s\gamma_{jk}=2h_{jk},$ it follows that
\[
(\partial_s\gamma^{ij})\gamma_{jk}
+\gamma^{ij}(2h_{jk})=0,
\]
hence
\[
\partial_s\gamma^{ij}
=-2\gamma^{ik}h_{k\ell}\gamma^{\ell j}.
\]
Therefore
\[
\partial_sH=(\partial_s\gamma^{ij})h_{ij}+\gamma^{ij}\partial_sh_{ij}=-2\gamma^{ik}h_{k\ell}\gamma^{\ell j}h_{ij}+\gamma^{ij}\langle R(\nu,X_i)\nu,X_j\rangle_g+\gamma^{ij}\langle A(X_i),A(X_j)\rangle_g.
\]
The first term, $-2\gamma^{ik}h_{k\ell}\gamma^{\ell j}h_{ij},$
represents the variation of the inverse metric $\gamma^{ij}$ contracted with the second fundamental form $h_{ij}$. Raising indices, the tensor $\gamma^{ik}h_{k\ell}$ can be identified with the $(1,1)$-tensor $A^{i}{}_{\ell}$, while $\gamma^{\ell j}h_{ij}$ corresponds to $A^{\ell}{}_{i}$. Hence the above expression equals
$-2\operatorname{tr}(A^2).$ The third term,
$\gamma^{ij}\langle A(X_i),A(X_j)\rangle_g,$
contracts to
$\operatorname{tr}(A^2).$
 Since $|A|^2=\operatorname{tr}(A^2)$ for symmetric operators, these contributions combine to give $-|A|^2,$
which governs the extrinsic part of the evolution equation for the mean curvature.
For the curvature term, our convention yields
\[
\gamma^{ij} \langle R(\nu, X_i) \nu, X_j \rangle_g = -\gamma^{ij} \langle R(X_i, \nu) \nu, X_j \rangle_g=-\mathrm{Ric}_g(\nu,\nu).
\]
 Hence
\begin{equation}\label{Mean curvature}
\partial_sH+|A|^2+\mathrm{Ric}_g(\nu,\nu)=0.
\end{equation}
Differentiating $\partial_su=\frac12Hu$ once more and substituting the evolution equation for $H$ gives
\[
\partial_s^2u=\Bigl(\frac12\partial_sH+\frac14H^2\Bigr)u=\Bigl(-\frac12|A|^2-\frac12\mathrm{Ric}_g(\nu,\nu)+\frac14H^2\Bigr)u.
\]
If $\lambda_1,\lambda_2$ are the principal curvatures of $\Sigma_{a+s}$, then $H=\lambda_1+\lambda_2$ and $|A|^2=\lambda_1^2+\lambda_2^2$, so
\[
-\frac12|A|^2+\frac14H^2=-\frac14(\lambda_1-\lambda_2)^2.
\]
Therefore, $\mathrm{Ric}_g\geq 0$ and $u>0$ imply that
\[
\partial_s^2u=-\frac12\mathrm{Ric}_g(\nu,\nu)u-\frac14(\lambda_1-\lambda_2)^2u\le 0
\]
on $\Sigma_a\times[0,T-a]$. Since $T>a$ was arbitrary, for each fixed $y\in\Sigma_a$ the function $s\mapsto u(y,s)$ is $C^2$ on $[0,\infty)$ and satisfies $\partial_s^2u(y,s)\le 0$ on every finite interval. Hence $u(y,\cdot)$ is concave on $[0,\infty)$.

We claim that $\partial_su(y,s)\ge 0$ for all $(y,s)\in\Sigma_a\times[0,\infty)$. Suppose for contradiction that $\partial_su(y,s_0)<0$ for some $s_0\ge 0$. Concavity of $u(y,\cdot)$ implies
\[
u(y,s)\le u(y,s_0)+\partial_su(y,s_0)(s-s_0)\qquad\text{for all }s\ge s_0.
\]
The right-hand side is negative for all sufficiently large $s$, contradicting $u>0$. Hence $\partial_su\ge 0$ everywhere. Since $u>0$ and $\partial_su=\frac12Hu$, it follows that
\[
H\ge 0\qquad\text{on }\Sigma_a\times[0,\infty),
\]
and in particular $H_t\ge 0$ on $\Sigma_t$ for every $t\ge a$.

Because $u(y,\cdot)$ is concave, the derivative $\partial_su(y,\cdot)$ is nonincreasing. Combined with $\partial_su\ge 0$,
\[
0\le\partial_su(y,s)\le\partial_su(y,0)\qquad\text{for all }(y,s)\in\Sigma_a\times[0,\infty).
\]
Since $u(\cdot,0)=1$, $\partial_su(y,0)=\frac12H_a(y).$
Consequently, we have the bounds$$0 \le \partial_s u(y,s) \le \frac{1}{2} H_a(y).$$On the two-dimensional surface $\Sigma_{a+s}$, the Cauchy–Schwarz inequality applied to the principal curvatures $\kappa_1, \kappa_2$ implies $|A|^2 \ge \frac{1}{2} H^2$, with equality if and only if the point is umbilical. It follows from the Gauss equation that$$G = H^2 - |A|^2 \le \frac{1}{2} H^2.$$Multiplying by the Jacobian $J$ and substituting the identity $2(\partial_s u)^2 = \frac{1}{2} H^2 J$, we obtain the pointwise estimate$$(G \circ F) J \le \frac{1}{2} H^2 J = 2 (\partial_s u)^2 \le \frac{1}{2} H_a^2.$$
Define the constants
\[
C_G(a):=\frac12\int_{\Sigma_a} H_a^2\, dA_a,
\qquad
C_H(a):=\int_{\Sigma_a} H_a\, dA_a.
\]
These quantities are finite because $\Sigma_a$ is compact and $H_a$ is smooth.
Let
\[
E_{a,T}:=f^{-1}([a,T]),\qquad U_{a,T}:=f^{-1}((a,T)).
\]
The properness of $f$ ensures that the region $E_{a,T}$ is compact. Its interior, $U_{a,T}$, differs from $E_{a,T}$ only by the boundary $\partial E_{a,T} = \Sigma_a \cup \Sigma_T$. As these level sets are smooth embedded hypersurfaces of codimension one, they constitute a set of measure zero with respect to the Riemannian volume form $dV_g$. Consequently, for any integrable function $\varphi$ on $E_{a,T}$, we have:$$\int_{E_{a,T}} \varphi \, dV_g = \int_{U_{a,T}} \varphi \, dV_g.
$$

The restriction of $F$ yields a diffeomorphism from $\Sigma_a \times (0, T-a)$ onto $U_{a,T}$. The identity $f \circ F = a + s$ ensures that the image of this product domain is contained within $U_{a,T}$, while the global diffeomorphism property of $F$ guarantees that every point in $U_{a,T}$ is uniquely represented. Under this map, the volume form pulls back according to the identity$$F^*(dV_g) = J \, dA_a \, ds.$$ Indeed, this follows because $\partial_s F = \nu$ is a unit vector field that remains orthogonal to the level sets $\Sigma_{a+s}$, which are themselves the images of the tangent spaces of $\Sigma_a$ under the flow. Therefore
\begin{equation}\label{G}
\begin{aligned}
\int_{E_{a,T}} G\,dV_g
&= \int_{U_{a,T}} G\, dV_g
= \int_{\Sigma_a \times (0, T-a)} (G \circ F)\, F^*(dV_g)  \\
&= \int_{\Sigma_a}\int_0^{T-a} (G \circ F)J\,ds\,dA_a
\le \int_{\Sigma_a}\int_0^{T-a}\tfrac12 H_a^2\,ds\,dA_a
= C_G(a)(T-a).
\end{aligned}
\end{equation}

Next observe that,  using \eqref{partial J},
\[
\partial_s(HJ)=(\partial_sH)J+H(\partial_sJ)=(\partial_sH+H^2)J=(G-\mathrm{Ric}_g(\nu,\nu))J.
\]
Integrating over $\Sigma_a\times(\varepsilon,T-a)$ for $\varepsilon\in(0,T-a)$ yields
\[
\int_{\Sigma_a}\int_\varepsilon^{T-a}(\mathrm{Ric}_g(\nu,\nu)-G)J\,ds\,dA_a=\int_{\Sigma_a}H(\cdot,\varepsilon)J(\cdot,\varepsilon)\,dA_a-\int_{\Sigma_a}H(\cdot,T-a)J(\cdot,T-a)\,dA_a.
\]
All quantities are continuous on the compact strip $\Sigma_a\times[0,T-a]$. Taking the limit as $\varepsilon \downarrow 0$ and invoking the dominated convergence theorem, justified by the continuity of the integrand and the compactness of the region, we obtain$$\int_{U_{a,T}} \left( \mathrm{Ric}_g(\nu,\nu) - G \right) dV_g = \int_{\Sigma_a} H_a \, dA_a - \int_{\Sigma_a} H(\cdot, T-a) J(\cdot, T-a) \, dA_a.$$We may replace the domain of integration $U_{a,T}$ with its closure $E_{a,T}$ as the boundary is a set of measure zero. Since the flow map $F_{T-a} : \Sigma_a \to \Sigma_T$ is a diffeomorphism, the curvature and the area element transform as$$F_{T-a}^* H_T = H(\cdot, T-a) \quad \text{and} \quad F_{T-a}^* dA_T = J(\cdot, T-a) \, dA_a.$$By the naturality of the pullback, it follows that$$F_{T-a}^*(H_T \, dA_T) = H(\cdot, T-a) J(\cdot, T-a) \, dA_a,$$which implies the integral identity$$\int_{\Sigma_T} H_T \, dA_T = \int_{\Sigma_a} F_{T-a}^*(H_T \, dA_T) = \int_{\Sigma_a} H(\cdot, T-a) J(\cdot, T-a) \, dA_a.$$

Since $H_T\ge 0,$ we obtain
$$
\int_{E_{a,T}}(\mathrm{Ric}_g(\nu,\nu)-G)\,dV_g\le C_H(a).
$$
Adding the two shell estimates and using \eqref{G}, we obtain
\begin{equation}\label{Ric}
\int_{E_{a,T}} \mathrm{Ric}_g(\nu,\nu)\, dV_g
\le C_G(a)(T-a) + C_H(a).
\end{equation}

\textit{Step 6.} We estimate the scalar curvature integral over the shell.
Let $K_t$ denote the Gaussian curvature of the surface $\Sigma_t$. In dimension $3$, the Gauss equation implies:$$K_t = \sec_g(e_1, e_2) + \det(A_t).$$ Substituting the identity $\det(A_t) = \frac{1}{2} G_t$ and the Ricci trace $$\mathrm{Sc}_g = 2\sec_g(e_1, e_2) + 2\mathrm{Ric}_g(\nu, \nu),$$ we eliminate the sectional curvature term to obtain the identity:$$\mathrm{Sc}_g = 2K_t + 2\mathrm{Ric}_g(\nu, \nu) - G_t \qquad \text{on } \Sigma_t.$$

On $U_{a,T}$ every point is of the form $F(y,s)$ with $0<s<T-a$, so it lies on $\Sigma_{a+s}$. Therefore
\[
\mathrm{Sc}_g(F(y,s))=2K_{a+s}(F(y,s))+2\mathrm{Ric}_g(\nu,\nu)(F(y,s))-G(F(y,s)).
\]
Multiplying by $J(y,s)$ and integrating over $\Sigma_a\times(0,T-a)$, with $F^*(dV_g)=J\,dA_a\,ds$, yields
\[
\int_{U_{a,T}}\mathrm{Sc}_g\,dV_g=2\int_{\Sigma_a}\int_0^{T-a}K_{a+s}(F(y,s))J\,ds\,dA_a+\int_{U_{a,T}}(\mathrm{Ric}_g(\nu,\nu)-G)\,dV_g+\int_{U_{a,T}}\mathrm{Ric}_g(\nu,\nu)\,dV_g.
\]
For each fixed $s \in (0, T-a)$, the diffeomorphism $F_s : \Sigma_a \to \Sigma_{a+s}$ preserves the Euler characteristic, so $\chi(\Sigma_{a+s}) = \chi(\Sigma_a)$. By the Gauss-Bonnet theorem and the change of variables formula:$$\int_{\Sigma_{a+s}} K_{a+s} \, dA_{a+s} = \int_{\Sigma_a} K_{a+s}(F(y,s)) J(y,s) \, dA_a = 2\pi \chi(\Sigma_a).$$Integrating this identity over the flow parameter $s \in [0, T-a]$ yields:$$2 \int_0^{T-a} \int_{\Sigma_a} K_{a+s}(F(y,s)) J(y,s) \, dA_a \, ds = 2 \int_0^{T-a} 2\pi \chi(\Sigma_a) \, ds = 4\pi \chi(\Sigma_a)(T-a).$$
Since $\Sigma_a$ is a connected, closed, orientable surface, as shown in Step~4, we have $\chi(\Sigma_a)\le 2$. Hence
\[
4\pi \chi(\Sigma_a)(T-a)\le 8\pi (T-a).
\]
Extending the integrals to $E_{a,T}$ and using \eqref{Ric}, we obtain
\[
\int_{E_{a,T}} \mathrm{Sc}_g \, dV_g
\le (8\pi + C_G(a))(T-a) + 2C_H(a)
\qquad \text{for all } T>a .
\]
\textit{Step 7.} We compare with metric balls to finish the proof.
We first prove that
\[
f(x)-a=d_g(x,C_a)\qquad\text{for all }x\in U_a.
\]
 To prove the inequality $d_g(x, C_a) \le f(x) - a$, let $x$ be an arbitrary point in $U_a$. By the surjectivity of the flow map $F$, there exist $y \in \Sigma_a$ and $s > 0$ such that $x = F(y, s)$, which implies $f(x) = a + s$. The integral curve $\gamma_y : [0, s] \to M$ defined by $\gamma_y(\tau) = F(y, \tau)$ is a unit-speed path connecting $y \in \Sigma_a \subseteq C_a$ to $x$. Since the length of this curve is exactly $s = f(x) - a$, it follows from the definition of the distance function to a set that $d_g(x, C_a) \le f(x) - a$.
 
For the reverse inequality, let $c:[0,b]\to M$ be any piecewise $C^1$ curve with $c(0)=x$ and $c(b)\in C_a$. Define
$\tau:=\inf\{s\in[0,b]:c(s)\in C_a\}.$ Since $C_a$ is closed and $c$ is continuous, $c(\tau)\in C_a$. For $0\le s<\tau$ we have $c(s)\notin C_a$, so $c(s)\in U_a$. Thus $c([0,\tau))\subset U_a$. By continuity and recalling \eqref{closer}, namely that $\overline{U_a}^{\,M}=E_a$, we obtain
$c(\tau)\in E_a .$
Therefore
\[
c(\tau)\in E_a\cap C_a=E_a\cap f^{-1}([0,a])=f^{-1}(a)=\Sigma_a\subset E^\circ.
\]
Hence $c([0,\tau])\subset E^\circ$ and $f(c(\tau))=a$. Let
$0=t_0<t_1<\cdots<t_N=\tau$
be a partition such that $c$ is $C^1$ on each $[t_{j-1},t_j]$. Since $c([0,\tau])\subset E^\circ$, the image segment $c([0,\tau])$ is compact in $E^\circ$. The restriction of $f$ to an open neighborhood of this compact set is smooth (hence locally Lipschitz), so $f\circ c$ is absolutely continuous on each $[t_{j-1},t_j]$ and the chain rule
\[
\frac{d}{ds}(f\circ c)(s)=df_{c(s)}(c'(s))
\]
holds for almost every $s\in[t_{j-1},t_j]$. Since $|\nabla f|_g=1$ on $E^\circ$, for each $j$, we obtain
\[
f(c(t_{j-1}))-f(c(t_j))=-\int_{t_{j-1}}^{t_j}df_{c(s)}(c'(s))\,ds\le\int_{t_{j-1}}^{t_j}|\nabla f|_g\,|c'(s)|_g\,ds=\int_{t_{j-1}}^{t_j}|c'(s)|_g\,ds.
\]
 Summing over $j=1,\dots,N$, we have 
\[
f(x)-f(c(\tau))\le\int_0^\tau|c'(s)|_g\,ds\le L_g(c).
\]
Since $f(c(\tau))=a$, we get $f(x)-a\le L_g(c).$
Taking the infimum over all piecewise $C^1$ curves from $x$ to $C_a$ yields $f(x)-a\le d_g(x,C_a)$. Combined with the opposite inequality, we prove that 
\[
f(x)-a=d_g(x,C_a)\qquad\text{for all }x\in U_a.
\]

Now fix $p\in M$ and $r>0$. The set $C_a$ is compact, so $d_g(p,C_a)<\infty$. Let $x\in B_p(r)\cap U_a$. By the triangle inequality for distance to a set,
\[
f(x)-a=d_g(x,C_a)\le d_g(x,p)+d_g(p,C_a)<r+d_g(p,C_a).
\]
Hence
\[
f(x)<a+r+d_g(p,C_a),
\]
so
\[
B_p(r)\cap U_a\subset E_{a,\,a+r+d_g(p,C_a)}.
\]
Since $\mathrm{Sc}_g\ge 0$ on $M$,
\[
\int_{B_p(r)} \mathrm{Sc}_g \, dV_g = \int_{B_p(r) \cap C_a} \mathrm{Sc}_g \, dV_g + \int_{B_p(r) \cap U_a} \mathrm{Sc}_g \, dV_g\le\int_{C_a}\mathrm{Sc}_g\,dV_g+\int_{E_{a,\,a+r+d_g(p,C_a)}}\mathrm{Sc}_g\,dV_g.
\]
Define
\[
C_0(a):=\int_{C_a}\mathrm{Sc}_g\,dV_g,
\]
which is finite since $C_a$ is compact and $\mathrm{Sc}_g$ is continuous. Apply the shell estimate of Step 6 with
\[
T=a+r+d_g(p,C_a)>a:
\]
\[
\int_{E_{a,\,a+r+d_g(p,C_a)}}\mathrm{Sc}_g\,dV_g\le\bigl(8\pi+C_G(a)\bigr)(r+d_g(p,C_a))+2C_H(a).
\]
Therefore
\[
\int_{B_p(r)}\mathrm{Sc}_g\,dV_g\le C_0(a)+(8\pi+C_G(a))r+\bigl(8\pi+C_G(a)\bigr)d_g(p,C_a)+2C_H(a).
\]
Set
\[
A_{p,a}:=C_0(a)+\bigl(8\pi+C_G(a)\bigr)d_g(p,C_a)+2C_H(a),\qquad B_a:=8\pi+C_G(a).
\]
Then
\[
\int_{B_p(r)}\mathrm{Sc}_g\,dV_g\le A_{p,a}+B_a r\qquad\text{for all }r>0.
\]
Dividing by $r$ and taking the $\limsup$ as $r\to\infty$ yields
\[
\limsup_{r\to\infty}\frac{1}{r}\int_{B_p(r)}\mathrm{Sc}_g\,dV_g\le B_a=8\pi+C_G(a)<\infty.
\]
Since $p\in M$ and $a>0$ were arbitrary, the theorem is proved.
\end{proof}

\begin{remark}
The identity $f(x) - a = d_g(x, C_a)$ for $x \in U_a$ shows that $f$ coincides with the Riemannian distance function to the set $C_a$ along the integral curves of $\nu$. However, we cannot choose $a$ such that the level set $\Sigma_a$ is minimal. Indeed, a theorem of Meeks, Simon, and Yau \cite[Theorem~6]{zbMATH03824612} states that if a complete noncompact orientable $3$-manifold $N$ with nonnegative Ricci curvature contains a compact embedded minimal surface $\Sigma$, then $N$ is isometric to a product $\Sigma \times \mathbb{R}$. Since $M \cong \mathbb{R}^3$  has only one end, such an isometric splitting is topologically precluded. Consequently, the mean curvature $H_a$ cannot vanish identically, and the boundary term $C_G(a)$ must be accounted for in the integral estimate. Furthermore, in our general theorem one only obtains the finiteness of $C_G(a)$ from the compactness of $\Sigma_a$, not a uniform bound $C_G(a) \le 8\pi$.
\end{remark}

\begin{proof}[Proof of Theorem~\ref{3D Cohn}]
Proposition~\ref{avr} implies Theorem~\ref{3D Cohn}~(I). For the proof of Theorem~\ref{3D Cohn}~(II), we use the theorem of Schoen--Yau--Liu. If \(M\) is not diffeomorphic to \(\mathbb{R}^3\), then Proposition~\ref{codim2-splitting-Sc} implies that \(\mathcal{C}(g)\) is finite. If \(M\) is diffeomorphic to \(\mathbb{R}^3\), then Proposition~\ref{level set} implies Theorem~\ref{3D Cohn}~(II).
\end{proof}

\section{Weighted Cohn--Vossen-Type Inequalities}\label{weighted}
This section extends the Cohn--Vossen-type inequalities
from Riemannian manifolds to weighted Riemannian manifolds.

The $m$-Bakry--Émery  curvature on a weighted Riemannian manifold $  (M^n,g,e^{-f}\,dV_g)  $ is defined by
$$\mathrm{Ric}_{f,m} := \mathrm{Ric}_g + \nabla^2 f - \frac{1}{m}\,df \otimes df, \qquad m \in (0,\infty],$$
where the last term is omitted when $m=\infty $. This tensor plays a central role in the development of synthetic notions of Ricci curvature on 
RCD spaces. For the definition of RCD spaces, see \cite{zbMATH07740333}.

 Motivated by this circle of ideas, the author introduces  the weighted scalar curvature 
\[
  \mathrm{Sc}_{\alpha,\beta}
  :=
  \mathrm{Sc}_g+\alpha\Delta f-\beta|\nabla f|^2,
\]
where $\alpha,\beta\in\mathbb R$, and studies its properties on weighted
Riemannian manifolds in
\cite[Section~4]{zbMATH07342230}.

Throughout, $(M^n,g)$ is complete, connected, noncompact, and $n>2$.
We fix a base point $p\in M$ and write $B_r:=B_g(p,r)$.
The Laplacian convention is $\Delta=\operatorname{div}\nabla$, equivalently
$\Delta f=\operatorname{tr}_g(\nabla^2 f)$, so that for every compactly
supported Lipschitz function $\psi$ and every $f\in C^\infty(M)$,
\[
  \int_M \psi\,\Delta f\,dV_g
  =
  -\int_M \langle\nabla\psi,\nabla f\rangle\,dV_g.
\]
With this sign convention, $\mathrm{Sc}_{\alpha,\beta}$ agrees with the
weighted scalar curvature introduced in~\cite[Section~4]{zbMATH07342230}.
The arguments below use only the displayed formula for
$\mathrm{Sc}_{\alpha,\beta}$.

We study complete smooth Riemannian metrics $g$ on connected smooth manifolds
$M^n$ ($n\geq 3$) for which Yau's problem admits a positive answer.
Precisely, we assume
\begin{equation}\label{eq:base-assumptions}
  \mathrm{Ric}_g\geq 0,
  \qquad
  A_p:=\limsup_{r\to\infty}
  r^{2-n}\int_{B(p,r)}\mathrm{Sc}_g\,dV_g<\infty.
\end{equation}
Since $\mathrm{Ric}_g\ge 0$, one has $\mathrm{Sc}_g\ge 0$, so for every
$\varepsilon>0$ there exists $r_\varepsilon$ such that
\begin{equation}\label{eq:scalar-growth-eps}
  \int_{B_r}\mathrm{Sc}_g\,dV_g
  \le (A_p+\varepsilon)\,r^{n-2}
  \qquad(r\ge r_\varepsilon).
\end{equation}
All estimates are asserted for sufficiently large $r$; compact subsets play
no role in the asymptotic bounds. For convenience, we introduce  notation for the trace of the Bakry--\'Emery 
tensor, which will be used throughout the estimates:
\[
  Q_m := \operatorname{tr}_g(\operatorname{Ric}_{f,m})
  = \mathrm{Sc}_g + \Delta f - \frac{1}{m}|\nabla f|^2
\]
for $m<\infty$, and similarly $Q_\infty := \operatorname{tr}_g(\operatorname{Ric}_{f,\infty}) = \mathrm{Sc}_g + \Delta f$ for $m=\infty$.
Note that the assumption $\operatorname{Ric}_{f,m} \geq 0$ trivially implies $Q_m \geq 0$.

\begin{theorem}[Finite-dimensional Bakry--\'Emery tensor]
\label{thm:finite-m}
Let $(M^n,g,e^{-f}\,dV_g)$ ($n\geq 3$) be a weighted Riemannian manifold
satisfying $\operatorname{Ric}_{f,m}\geq 0$ for some $m<\infty$, and suppose
$g$ satisfies~\eqref{eq:base-assumptions}.
Then for every $\alpha,\beta\in\mathbb{R}$ and every $\varepsilon>0$ there
exists $r_0=r_0(\varepsilon)$ such that for all $r\geq r_0$,
\begin{equation}\label{eq:finite-absolute-main}
  \int_{B_g(p,r)}|\operatorname{Sc}_{\alpha,\beta}|\,dV_g
  \leq C(n,m,\alpha,\beta)(A_p+\varepsilon+1)\,r^{n-2}.
\end{equation}
\end{theorem}

The strategy of the proof relies on localizing the problem using families of smooth cutoff functions and applying integration by parts. By tracing the Bakry--\'Emery condition $\operatorname{Ric}_{f,m}\geq 0$, we establish coercive $L^1$ bounds on both the gradient $|\nabla f|^2$ and the scalar curvature deficit $Q_m$. Substituting these bounds into the algebraic decomposition of $\mathrm{Sc}_{\alpha,\beta}$ then yields the desired absolute integral estimate.

\begin{proof}
Choose $\eta\in C^\infty([0,\infty))$ with $0\le\eta\le 1$,
$\eta\equiv 1$ on $[0,1]$, $\eta\equiv 0$ on $[3/2,\infty)$, and
$|\eta'|\le C$, and set
\[
  \phi_r(x):=\eta\!\left(\frac{d_g(p,x)}{r}\right).
\]
Since the distance function is $1$-Lipschitz, $\phi_r$ is compactly supported
(Hopf--Rinow), satisfies $\phi_r\equiv 1$ on $B_r$,
$\operatorname{supp}\phi_r\subset\overline B_{3r/2}\subset B_{2r}$, and
$|\nabla\phi_r|\le C/r$ a.e.
Define
\[
  I_r:=\int_M\phi_r^2|\nabla f|^2\,dV_g,\quad
  S_r:=\int_M\phi_r^2\mathrm{Sc}_g\,dV_g,\quad
  J_r:=\int_M|\nabla\phi_r|^2\,dV_g.
\]
For $2r\ge r_\varepsilon$, using~\eqref{eq:scalar-growth-eps} and the
Bishop--Gromov bound $\operatorname{Vol}_g(B_{2r})\le\omega_n(2r)^n$,
\begin{equation}\label{eq:SJ-growth}
  S_r\le C_1(n)(A_p+\varepsilon)\,r^{n-2},
  \qquad
  J_r\le C_2(n)\,r^{n-2}.
\end{equation}
Since $\phi_r^2\in W^{1,\infty}_c(M)$, integration by parts gives
\begin{equation}\label{eq:weak-ibp-cutoff}
  \int_M\phi_r^2\,\Delta f\,dV_g
  =-2\int_M\phi_r\langle\nabla\phi_r,\nabla f\rangle\,dV_g.
\end{equation}

Since $\operatorname{Ric}_{f,m}\ge 0$ implies $Q_m\ge 0,$ we have
 $m^{-1}|\nabla f|^2\le\mathrm{Sc}_g+\Delta f$.  
Multiplying both sides by $\phi_r^2$ and integrating over $M$ gives
\[
  \frac{1}{m}\int_M \phi_r^2|\nabla f|^2\,dV_g
  \le \int_M \phi_r^2\,\mathrm{Sc}_g\,dV_g
    + \int_M \phi_r^2\,\Delta f\,dV_g
  = S_r + \int_M \phi_r^2\,\Delta f\,dV_g.
\]
Applying the integration by parts formula~\eqref{eq:weak-ibp-cutoff} to the
last term,
\[
  \frac{1}{m}\,I_r
  \le S_r - 2\int_M \phi_r\langle\nabla\phi_r,\nabla f\rangle\,dV_g
  \le S_r + 2\Bigl|\int_M \phi_r\langle\nabla\phi_r,\nabla f\rangle\,dV_g\Bigr|.
\]
By the Cauchy--Schwarz inequality,
\[
  \Bigl|\int_M \phi_r\langle\nabla\phi_r,\nabla f\rangle\,dV_g\Bigr|
  \le \left(\int_M \phi_r^2|\nabla f|^2\,dV_g\right)^{1/2}
     \left(\int_M |\nabla\phi_r|^2\,dV_g\right)^{1/2}
  = I_r^{1/2}\,J_r^{1/2}.
\]
Hence
\[
  \frac{1}{m}\,I_r \le S_r + 2\,I_r^{1/2}\,J_r^{1/2}.
\]
Since  $2\,I_r^{1/2}\,J_r^{1/2}
\le \tfrac{1}{2m}\,I_r + 2m\,J_r$, we have 
\[
  \frac{1}{m}\,I_r
  \le S_r + \frac{1}{2m}\,I_r + 2m\,J_r.
\]
Subtracting $\tfrac{1}{2m}I_r$ from both sides and multiplying through by $2m$,
\[
  I_r \le 2m\,S_r + 4m^2 J_r.
\]
Substituting the bounds from~\eqref{eq:SJ-growth},
\begin{align}
  I_r
  &\le 2m \cdot C_1(n)(A_p+\varepsilon)\,r^{n-2}
     + 4m^2 \cdot C_2(n)\,r^{n-2} \notag\\
  &= \bigl[2m\,C_1(n)(A_p+\varepsilon) + 4m^2 C_2(n)\bigr]r^{n-2}. \notag
\end{align}
Since $(A_p+\varepsilon) \le (A_p+\varepsilon+1)$ and $1 \le (A_p+\varepsilon+1)$,
both terms are absorbed by setting
$C(n,m) := 2m\,C_1(n)+4m^2 C_2(n)$, giving
\begin{equation}\label{eq:finite-grad-bound}
  I_r \le 2m\,S_r + 4m^2 J_r
     \le C(n,m)(A_p+\varepsilon+1)\,r^{n-2}.
\end{equation}
Since $\phi_r \equiv 1$ on $B_r$, this proves~\eqref{eq:finite-grad-bound}
for $\int_{B_r}|\nabla f|^2\,dV_g$.
Since $Q_m \ge 0$ and $\phi_r \equiv 1$ on $B_r$, extending the domain
of integration and expanding the definition of $Q_m$,
\[
  \int_{B_r} Q_m\,dV_g
  = \int_{B_r}\phi_r^2\,Q_m\,dV_g
  \le \int_M \phi_r^2\,Q_m\,dV_g
  = S_r
    - 2\int_M\phi_r\langle\nabla\phi_r,\nabla f\rangle\,dV_g
    - \frac{1}{m}\,I_r,
\]
where~\eqref{eq:weak-ibp-cutoff} was used for the $\Delta f$ term.
Discarding $-\tfrac{1}{m}I_r \le 0$ and applying Cauchy--Schwarz,
\[
  \int_{B_r} Q_m\,dV_g
  \le S_r + 2\,I_r^{1/2}J_r^{1/2}.
\]
 It remains to bound $I_r^{1/2}J_r^{1/2}$. By~\eqref{eq:finite-grad-bound}
and the bound $J_r \le C_2(n)\,r^{n-2}$ from~\eqref{eq:SJ-growth},
\[
  I_r^{1/2}\,J_r^{1/2}
  \le \bigl[C(n,m)(A_p+\varepsilon+1)\,r^{n-2}\bigr]^{1/2}
     \cdot \bigl[C_2(n)\,r^{n-2}\bigr]^{1/2}
  = C(n,m)^{1/2}\,C_2(n)^{1/2}\,(A_p+\varepsilon+1)^{1/2}\,r^{n-2}.
\]
Since $(A_p+\varepsilon+1)^{1/2} \le (A_p+\varepsilon+1)$, we obtain
\[
  2\,I_r^{1/2}\,J_r^{1/2}
  \le C'(n,m)(A_p+\varepsilon+1)\,r^{n-2},
\]
where $C'(n,m) := 2\,C(n,m)^{1/2}C_2(n)^{1/2}$.
Substituting back and applying $S_r \le C_1(n)(A_p+\varepsilon)\,r^{n-2}$
from~\eqref{eq:SJ-growth},
\[
  \int_{B_r}Q_m\,dV_g
  \le S_r + 2\,I_r^{1/2}J_r^{1/2}
  \le \bigl[C_1(n)(A_p+\varepsilon) + C'(n,m)(A_p+\varepsilon+1)\bigr]r^{n-2}.
\]
Both terms are absorbed by $(A_p+\varepsilon+1)$, giving
\begin{equation}\label{eq:finite-Q-bound}
  \int_{B_r}Q_m\,dV_g
  \le C(n,m)(A_p+\varepsilon+1)\,r^{n-2},
\end{equation}
where $C(n,m)$ is updated to $\max\{C_1(n), C'(n,m)\} + C'(n,m)$,
depending only on $n$ and $m$.

The algebraic identity
\begin{equation}\label{eq:finite-algebraic-identity}
  \mathrm{Sc}_{\alpha,\beta}
  =(1-\alpha)\,\mathrm{Sc}_g
   +\alpha\,Q_m
   +\!\left(\frac{\alpha}{m}-\beta\right)|\nabla f|^2
\end{equation}
follows by substituting $\alpha\,\Delta f=\alpha\,Q_m-\alpha\,\mathrm{Sc}_g
+(\alpha/m)|\nabla f|^2$ into the definition of $\mathrm{Sc}_{\alpha,\beta}$.
Since $\mathrm{Sc}_g\ge 0$ and $Q_m\ge 0$,
\[
  |\mathrm{Sc}_{\alpha,\beta}|
  \le|1-\alpha|\,\mathrm{Sc}_g
    +|\alpha|\,Q_m
    +\left|\frac{\alpha}{m}-\beta\right||\nabla f|^2.
\]
Integrating the pointwise bound
\[
  |\mathrm{Sc}_{\alpha,\beta}|
  \le |1-\alpha|\,\mathrm{Sc}_g
    + |\alpha|\,Q_m
    + \left|\frac{\alpha}{m}-\beta\right||\nabla f|^2
\]
over $B_r$ and splitting into three terms,
\[
  \int_{B_r}|\mathrm{Sc}_{\alpha,\beta}|\,dV_g
  \le |1-\alpha|\int_{B_r}\mathrm{Sc}_g\,dV_g
    + |\alpha|\int_{B_r}Q_m\,dV_g
    + \left|\frac{\alpha}{m}-\beta\right|\int_{B_r}|\nabla f|^2\,dV_g.
\]
We bound each term separately. For the first, \eqref{eq:scalar-growth-eps} gives
\[
  |1-\alpha|\int_{B_r}\mathrm{Sc}_g\,dV_g
  \le |1-\alpha|(A_p+\varepsilon)\,r^{n-2}.
\]
For the second, \eqref{eq:finite-Q-bound} gives
\[
  |\alpha|\int_{B_r}Q_m\,dV_g
  \le |\alpha|\,C(n,m)(A_p+\varepsilon+1)\,r^{n-2}.
\]
For the third, \eqref{eq:finite-grad-bound} gives
\[
  \left|\frac{\alpha}{m}-\beta\right|\int_{B_r}|\nabla f|^2\,dV_g
  \le \left|\frac{\alpha}{m}-\beta\right|C(n,m)(A_p+\varepsilon+1)\,r^{n-2}.
\]
Since $(A_p+\varepsilon) \le (A_p+\varepsilon+1)$, all three terms are
bounded by a constant multiple of $(A_p+\varepsilon+1)\,r^{n-2}$.
Summing and setting
\[
  C(n,m,\alpha,\beta)
  := |1-\alpha| + |\alpha|\,C(n,m)
     + \left|\frac{\alpha}{m}-\beta\right|C(n,m),
\]
which depends only on $n$, $m$, $\alpha$, and $\beta$, yields
\begin{equation*}
  \int_{B_r}|\mathrm{Sc}_{\alpha,\beta}|\,dV_g
  \le C(n,m,\alpha,\beta)(A_p+\varepsilon+1)\,r^{n-2}.\qedhere
\end{equation*}
\end{proof}

For the limiting Bakry--\'Emery tensor, set
$Q_\infty:=\mathrm{Sc}_g+\Delta f;$ tracing $\mathrm{Ric}_{f,\infty}\ge 0$ gives $Q_\infty\ge 0$.

\begin{theorem}[Infinite-dimensional Bakry--\'Emery tensor]
\label{thm:infinite-N}
Let $(M^n,g,e^{-f}\,dV_g)$ ($n\geq 3$)  be a weighted Riemannian manifold
satisfying $\operatorname{Ric}_{f,\infty}\ge 0$, and suppose $g$
satisfies~\eqref{eq:base-assumptions}.
\begin{enumerate}
\item[\textup{(a)}] If $\alpha\le 0$ and $\beta\ge 0$, then for every
$\varepsilon>0$ there exists $r_0=r_0(\varepsilon)$ such that for all
$r\ge r_0$,
\begin{equation}\label{eq:infty-signed-upper}
  \int_{B_g(p,r)}\mathrm{Sc}_{\alpha,\beta}\,dV_g
  \le(1-\alpha)\int_{B_g(p,r)}\mathrm{Sc}_g\,dV_g
  \le(1-\alpha)(A_p+\varepsilon)\,r^{n-2}.
\end{equation}

\item[\textup{(b)}] If $\beta>0$ and $\mathrm{Sc}_{\alpha,\beta}\ge 0$
pointwise, with $\alpha\in\mathbb{R}$ arbitrary, then for every $\varepsilon>0$
there exists $r_0=r_0(\varepsilon)$ such that for all $r\ge r_0$,
\begin{equation}\label{eq:infty-positive-beta}
  \int_{B_g(p,r)}\mathrm{Sc}_{\alpha,\beta}\,dV_g
  \le C(n,\alpha,\beta)(A_p+\varepsilon+1)\,r^{n-2}.
\end{equation}
\end{enumerate}
\end{theorem}

The sharpness of Theorem~\ref{thm:infinite-N} is established in
Remark~\ref{rem:no-absolute-infty} and Corollary~\ref{cor:sharp-range-infty}.
The two parts require different arguments. Part~\textup{(a)} is immediate
from the pointwise bound $\mathrm{Sc}_{\alpha,\beta}\le(1-\alpha)\,\mathrm{Sc}_g$
and~\eqref{eq:scalar-growth-eps}. In part~\textup{(b)}, the hypothesis
$\mathrm{Sc}_{\alpha,\beta}\ge 0$ replaces the coercivity of
$\operatorname{Ric}_{f,m}\ge 0$: testing against $\phi_r^2$ and applying
Cauchy--Schwarz and Young gives a Caccioppoli bound on $I_r$, from
which~\eqref{eq:infty-positive-beta} follows.

\begin{proof}
The cutoff $\phi_r$ and the quantities $I_r$, $S_r$, $J_r$ are as in
the proof of Theorem~\ref{thm:finite-m}, and the bounds~\eqref{eq:SJ-growth}
and~\eqref{eq:weak-ibp-cutoff} hold unchanged.

\textit{Part~\textup{(a)}.}
The algebraic identity
\[
  \mathrm{Sc}_{\alpha,\beta}
  =(1-\alpha)\,\mathrm{Sc}_g+\alpha\,Q_\infty-\beta|\nabla f|^2
\]
and the conditions $Q_\infty\ge 0$, $\alpha\le 0$, $\beta\ge 0$ together
give $\mathrm{Sc}_{\alpha,\beta}\le(1-\alpha)\,\mathrm{Sc}_g$ pointwise.
Integrating and applying~\eqref{eq:scalar-growth-eps}
proves~\eqref{eq:infty-signed-upper}. 

\textit{Part~\textup{(b)}.}
Since $\mathrm{Sc}_{\alpha,\beta}\ge 0$ and $\phi_r\equiv 1$ on $B_r$,
\[
  0
  \le\int_M\phi_r^2\,\mathrm{Sc}_{\alpha,\beta}\,dV_g
  =S_r
   -2\alpha\int_M\phi_r\langle\nabla\phi_r,\nabla f\rangle\,dV_g
   -\beta\,I_r.
\]

Rearranging and applying Cauchy--Schwarz to the cross term,
\[
  \beta\,I_r
  \le S_r + 2|\alpha|\,I_r^{1/2}J_r^{1/2}.
\]
Since $\beta>0$, one has 
$2|\alpha|\,I_r^{1/2}J_r^{1/2} \le \tfrac{\beta}{2}\,I_r
+ \tfrac{2\alpha^2}{\beta}\,J_r$, and hence
\[
  \beta\,I_r \le S_r + \frac{\beta}{2}\,I_r + \frac{2\alpha^2}{\beta}\,J_r.
\]
 Substituting
$S_r \le C_1(n)(A_p+\varepsilon)\,r^{n-2}$ and
$J_r \le C_2(n)\,r^{n-2}$ from~\eqref{eq:SJ-growth} gives
\[
  I_r
  \le \frac{2C_1(n)}{\beta}(A_p+\varepsilon)\,r^{n-2}
    + \frac{4\alpha^2 C_2(n)}{\beta^2}\,r^{n-2}.
\]
Since both terms are absorbed by $(A_p+\varepsilon+1)$, setting
$C(n,\alpha,\beta) :=\tfrac{2C_1(n)}{\beta} +
\tfrac{4\alpha^2 C_2(n)}{\beta^2}$ gives
\begin{equation}\label{eq:infty-grad-coercive}
  I_r
  \le \frac{2}{\beta}\,S_r + \frac{4\alpha^2}{\beta^2}\,J_r
  \le C(n,\alpha,\beta)(A_p+\varepsilon+1)\,r^{n-2}.
\end{equation}

Since $\phi_r\equiv 1$ on $B_r$ and $\mathrm{Sc}_{\alpha,\beta}\ge 0$,
extending the domain of integration and expanding as before,
\[
  \int_{B_r}\mathrm{Sc}_{\alpha,\beta}\,dV_g
  \le\int_M\phi_r^2\,\mathrm{Sc}_{\alpha,\beta}\,dV_g
  = S_r
    - 2\alpha\int_M\phi_r\langle\nabla\phi_r,\nabla f\rangle\,dV_g
    - \beta\,I_r
  \le S_r + 2|\alpha|\,I_r^{1/2}J_r^{1/2},
\]
where $-\beta\,I_r\le 0$ is discarded and Cauchy--Schwarz is applied to
the cross term. For the remaining term, \eqref{eq:infty-grad-coercive}
and $J_r\le C_2(n)\,r^{n-2}$ from~\eqref{eq:SJ-growth} give
\[
  I_r^{1/2}J_r^{1/2}
  \le \bigl[C(n,\alpha,\beta)(A_p+\varepsilon+1)\,r^{n-2}\bigr]^{1/2}
     \cdot\bigl[C_2(n)\,r^{n-2}\bigr]^{1/2}
  \le C'(n,\alpha,\beta)(A_p+\varepsilon+1)\,r^{n-2},
\]
where the last step uses $(A_p+\varepsilon+1)^{1/2}\le(A_p+\varepsilon+1)$.
Combining with $S_r\le C_1(n)(A_p+\varepsilon)\,r^{n-2}$
from~\eqref{eq:SJ-growth} and absorbing both terms
into $(A_p+\varepsilon+1)$ yields~\eqref{eq:infty-positive-beta}.
\end{proof}

\begin{remark}\label{rem:redundant}
The proof of part~\textup{(b)} does not use $\mathrm{Ric}_{f,\infty}\ge 0$;
it relies only on~\eqref{eq:base-assumptions}, the cutoff estimates, and the
pointwise condition $\mathrm{Sc}_{\alpha,\beta}\ge 0$.
The Bakry--\'Emery hypothesis is retained in the statement because it is
natural in the weighted-geometric setting and is used in part~\textup{(a)}.
\end{remark}

\begin{remark}\label{rem:no-absolute-infty}
The restrictions in Theorem~\ref{thm:infinite-N} are sharp. On
$(\mathbb{R}^n,g_{E})$ with $p=0$, the standing assumptions
hold with $A_p=0$, and the following examples show that no further estimates
are possible.

For  absolute-value estimates, if $\beta\ne 0$, take $f(x)=x_1$, giving
$\mathrm{Ric}_{f,\infty}=0$ and $\mathrm{Sc}_{\alpha,\beta}=-\beta$, so
$r^{2-n}\int_{B_r}|\mathrm{Sc}_{\alpha,\beta}|\,dx=|\beta|\omega_n r^2\to\infty$;
if $\beta=0$ and $\alpha\ne 0$, take $f(x)=\tfrac{1}{2}x_1^2$, giving
$\mathrm{Ric}_{f,\infty}\ge 0$ and $\mathrm{Sc}_{\alpha,0}=\alpha$, with
the same divergence. Thus~\eqref{eq:base-assumptions} and
$\mathrm{Ric}_{f,\infty}\ge 0$ alone yield no absolute-value estimate for
any $(\alpha,\beta)\ne(0,0)$.

For the nonnegative estimate: if $\beta<0$, the same $f(x)=x_1$ gives
$\mathrm{Sc}_{\alpha,\beta}=-\beta>0$ and
$r^{2-n}\int_{B_r}\mathrm{Sc}_{\alpha,\beta}\,dx=(-\beta)\omega_n r^2\to\infty$;
if $\beta=0$ and $\alpha>0$, the same $f(x)=\tfrac{1}{2}x_1^2$ gives
$\mathrm{Sc}_{\alpha,0}=\alpha>0$ with the same divergence. Thus the parameter range covered by Theorem~\ref{thm:infinite-N}\textup{(a)}--\textup{(b)} is optimal.
\end{remark}

\begin{corollary}[Optimal parameter range]\label{cor:sharp-range-infty}
Assume~\eqref{eq:base-assumptions}, $\mathrm{Ric}_{f,\infty}\ge 0$, and
$\mathrm{Sc}_{\alpha,\beta}\ge 0$. The range
\[
  (\alpha,\beta)\in
  \bigl(\mathbb{R}\times(0,\infty)\bigr)
  \cup
  \bigl((-\infty,0]\times\{0\}\bigr)
\]
is optimal for the estimate $\int_{B_g(p,r)}\mathrm{Sc}_{\alpha,\beta}\,dV_g=O(r^{n-2})$:
outside this range, Remark~\ref{rem:no-absolute-infty} supplies examples
satisfying all stated hypotheses for which
\[
  \limsup_{r\to\infty}r^{2-n}
  \int_{B_g(p,r)}\mathrm{Sc}_{\alpha,\beta}\,dV_g=+\infty.
\]
\end{corollary}

\begin{proof}
For $\beta>0$ the bound is Theorem~\ref{thm:infinite-N}\textup{(b)}.
For $\beta=0$ and $\alpha\le 0$, it follows from
Theorem~\ref{thm:infinite-N}\textup{(a)}, since $\mathrm{Sc}_{\alpha,0}\ge 0$
converts the one-sided estimate into a nonnegative bound.
The counterexamples for $\beta<0$ and $\beta=0$, $\alpha>0$ are supplied
by Remark~\ref{rem:no-absolute-infty}.
\end{proof}

\begin{proposition}[Necessity of the sign condition when $\beta>0$]
\label{prop:sign-needed-beta-positive}
Let $n>2$, $\alpha>0$, and $\beta>0$. The pointwise hypothesis
$\mathrm{Sc}_{\alpha,\beta}\ge 0$ in Theorem~\ref{thm:infinite-N}\textup{(b)}
cannot be omitted: there exists a smooth convex function $f$ on
$(\mathbb{R}^n,g_{E})$ with $\mathrm{Ric}_{f,\infty}\ge 0$
for which $\mathrm{Sc}_{\alpha,\beta}$ changes sign and
\[
  \limsup_{r\to\infty}r^{2-n}
  \int_{B_r(0)}\bigl(\alpha\,\Delta f-\beta|\nabla f|^2\bigr)\,dx=+\infty.
\]
\end{proposition}

\begin{proof}
We work on \((\mathbb R^n,g_{E})\) with base point \(0\). Then
\(\mathrm{Ric}_g=0\), \(\mathrm{Sc}_g=0\), and the standing assumptions hold
with \(A_p=0\). Thus
\[
  \mathrm{Sc}_{\alpha,\beta}
  =
  \alpha\Delta f-\beta|\nabla f|^2 .
\]

We first record the radial identities. Suppose that
\(v\in C^\infty([0,\infty))\), \(v\ge0\), \(v'\ge0\), and \(v\equiv0\) near
\(0\). Define
\[
  F(r):=\int_0^r v(s)\,ds,
  \qquad
  f(x):=F(|x|).
\]
Then \(f\) is smooth and convex on \(\mathbb R^n\). Indeed, for \(r=|x|>0\),
\[
  \nabla^2 f
  =
  v'(r)\,dr\otimes dr
  +
  \frac{v(r)}{r}
  \bigl(g_{E}-dr\otimes dr\bigr)\ge0,
\]
and \(f\) is smooth at the origin because \(v\equiv0\) near \(0\). Hence
\[
  \mathrm{Ric}_{f,\infty}
  =
  \nabla^2 f
  \ge0.
\]
Moreover,
\[
  \Delta f=v'(r)+\frac{n-1}{r}v(r),
  \qquad
  |\nabla f|^2=v(r)^2 .
\]
Writing \(\omega_{n-1}=|\mathbb S^{n-1}|\), the divergence theorem gives
\begin{equation}\label{eq:Phi-short-final}
  \Phi(R):=
  R^{2-n}
  \int_{B_R(0)}
  \bigl(\alpha\Delta f-\beta|\nabla f|^2\bigr)\,dx
  =
  \omega_{n-1}
  \left[
    \alpha Rv(R)
    -
    \beta R^{2-n}\int_0^R v(s)^2s^{n-1}\,ds
  \right].
\end{equation}
The first term is a boundary flux depending only on the current value
\(v(R)\), while the second term is the accumulated \(L^2\)-cost of the
gradient. We exploit this by keeping \(v\) at the old height \(a_{j-1}\) until
just before \(R_j\), and then raising it to the new height \(a_j\) in a thin
shell near \(R_j\).

Set \(R_j:=3j\) and \(a_0:=0\). Choose \(a_j>a_{j-1}\) inductively so that
\begin{equation}\label{eq:aj-final-short}
  \frac{\alpha}{2}a_jR_j
  \ge
  \frac{\beta}{n}a_{j-1}^2R_j^2+j,
  \qquad
  a_j>\frac{\alpha(n-1)}{\beta R_j}.
\end{equation}
Both conditions are lower bounds on \(a_j\), so any sufficiently large \(a_j\)
satisfies them. The first inequality makes the new boundary flux dominate the
old plateau cost by the margin \(j\); the second is used only to force
negativity of \(\mathrm{Sc}_{\alpha,\beta}\) at \(|x|=R_j\).

Choose
\[
  0<\delta_j<
  \min\left\{\frac12,\frac{\alpha}{2\beta a_j}\right\}.
\]
The bound \(\delta_j<1/2\) keeps the shells
\([R_j-\delta_j,R_j]\) pairwise disjoint, since consecutive radii differ by
\(3\). The bound \(\delta_j<\alpha/(2\beta a_j)\) controls the new-shell cost.

Let \(\theta\in C^\infty(\mathbb R)\) be nondecreasing, with
\(\theta\equiv0\) on \((-\infty,0]\) and \(\theta\equiv1\) on
\([1,\infty)\), and define
\[
  v(r):=
  \sum_{j=1}^{\infty}
  (a_j-a_{j-1})
  \theta\!\left(\frac{r-(R_j-\delta_j)}{\delta_j}\right).
\]
The sum is locally finite. Hence \(v\in C^\infty([0,\infty))\), \(v\ge0\),
\(v'\ge0\), and \(v\equiv0\) near \(0\). Therefore the radial function
\(f(x)=F(|x|)\), with \(F(r)=\int_0^r v(s)\,ds\), is smooth and convex on
\(\mathbb R^n\).

At \(R_j\), the first \(j\) layers have completed and no later layer has
begun; the sum therefore telescopes. Thus
\[
  v(R_j)=\sum_{k=1}^j(a_k-a_{k-1})=a_j,
  \qquad
  v'(R_j)=0,
  \qquad
  v(s)\le a_{j-1}\quad\text{for }0\le s\le R_j-\delta_j.
\]
Splitting the cost integral at \(R_j-\delta_j\), and using \(v\le a_{j-1}\)
before the shell and \(v\le a_j\) on the shell, gives
\[
  \int_0^{R_j}v(s)^2s^{n-1}\,ds
  \le
  a_{j-1}^2\int_0^{R_j}s^{n-1}\,ds
  +
  a_j^2\delta_jR_j^{n-1}
  =
  \frac{a_{j-1}^2R_j^n}{n}
  +
  a_j^2\delta_jR_j^{n-1}.
\]
Substituting this estimate into \eqref{eq:Phi-short-final} gives
\[
\begin{aligned}
  \Phi(R_j)
  &\ge
  \omega_{n-1}
  \left[
    \alpha a_jR_j
    -
    \frac{\beta}{n}a_{j-1}^2R_j^2
    -
    \beta a_j^2\delta_jR_j
  \right]  \\
  &\ge
  \omega_{n-1}
  \left[
    \frac{\alpha}{2}a_jR_j
    -
    \frac{\beta}{n}a_{j-1}^2R_j^2
  \right]
  \ge
  \omega_{n-1}j.
\end{aligned}
\]
Here the second line uses
$\beta a_j^2\delta_jR_j\le \frac{\alpha}{2}a_jR_j,$
and the last inequality uses \eqref{eq:aj-final-short}. Hence
\[
  \limsup_{R\to\infty}
  R^{2-n}
  \int_{B_R(0)}
  \bigl(\alpha\Delta f-\beta|\nabla f|^2\bigr)\,dx
  =
  +\infty .
\]

Finally, since \(\Phi(R_j)>0\), the continuous function
\(\mathrm{Sc}_{\alpha,\beta}\) is positive somewhere in \(B_{R_j}(0)\). On the
other hand, at \(|x|=R_j\), one has \(v(R_j)=a_j\) and \(v'(R_j)=0\), so
\[
  \mathrm{Sc}_{\alpha,\beta}
  =
  \alpha\left(v'(R_j)+\frac{n-1}{R_j}v(R_j)\right)
  -
  \beta v(R_j)^2
  =
  \frac{\alpha(n-1)a_j}{R_j}-\beta a_j^2<0
\]
by \eqref{eq:aj-final-short}. Therefore
\(\mathrm{Sc}_{\alpha,\beta}\) changes sign.
\end{proof}

\begin{proof}[Proof of Theorem~\ref{thm:intro-weighted}]
Parts~\textup{(i)}, \textup{(ii)}, and~\textup{(iii)} follow from
Theorems~\ref{thm:finite-m} and~\ref{thm:infinite-N}\textup{(a)}--\textup{(b)},
respectively.
\end{proof}

\addcontentsline{toc}{section}{\refname}
\bibliographystyle{alpha}
\bibliography{reference}

\end{document}